\documentclass[11pt]{amsart}
\usepackage{amsmath,amsthm,amsfonts,amscd,amssymb,eucal,latexsym,mathrsfs,enumitem,bm}
\usepackage[all]{xy}

\newtheorem{theorem}{Theorem}[section]
\newtheorem{corollary}[theorem]{Corollary}
\newtheorem{lemma}[theorem]{Lemma}
\newtheorem{proposition}[theorem]{Proposition}

\theoremstyle{definition}
\newtheorem{definition}[theorem]{Definition}
\newtheorem{remark}[theorem]{Remark}

\newtheorem{example}[theorem]{Example}
\newtheorem{question}[theorem]{Question}

\newcommand{\htopol}{h_{\text{\rm top}}}

\newcommand{\sep}{{\rm sep}}

\newcommand{\id}{{\rm id}}

\newcommand{\cT}{{\mathscr T}}

\newcommand{\cR}{{\mathscr R}}
\newcommand{\cP}{{\mathscr P}}

\newcommand{\cZ}{{\mathscr Z}}

\newcommand{\sL}{{\mathscr L}}

\newcommand{\sU}{{\mathscr U}}

\newcommand{\Cb}{\bm{\mathrm{C}}}

\newcommand{\eps}{\varepsilon}

\newcommand{\fA}{{\sf A}}
\newcommand{\fS}{{\sf S}}
\newcommand{\vN}{{\mathscr L}}
\newcommand{\C}{\bm{\mathrm{C}}}
\newcommand{\N}{\bm{\mathrm{N}}}
\newcommand{\Z}{\bm{\mathrm{Z}}}
\newcommand{\mc}[1]{\mathcal{#1}}

\DeclareMathOperator{\Sym}{Sym}

\DeclareMathOperator{\prob}{{\bf P}}

\allowdisplaybreaks

\begin{document}

\title[Dynamical alternating groups]{Dynamical alternating groups, stability, property Gamma, and inner amenability}

\author{David Kerr}
\address{David Kerr,
Mathematisches Institut,
WWU M{\"u}nster, 
Einsteinstr.\ 62, 
48149 M{\"u}nster, Germany}
\email{kerrd@uni-muenster.de}

\author{Robin Tucker-Drob}
\address{Robin Tucker-Drob,
Department of Mathematics,
Texas A\&M University,
College Station, TX 77843-3368, USA}
\email{rtuckerd@math.tamu.edu}

\date{March 8, 2021}

\begin{abstract}
We prove that the alternating group of a topologically free action of a countably infinite group $\Gamma$
on the Cantor set has the property that all of its $\ell^2$-Betti numbers vanish and, in the case that $\Gamma$ is amenable,
is stable in the sense of Jones and Schmidt and has property Gamma
(and in particular is inner amenable).
We show moreover in the realm of amenable $\Gamma$ that
there are many such alternating groups which are simple, finitely generated, and C$^*$-simple.
The device for establishing nonisomorphism among these examples is
a topological version of Austin's result on the invariance of measure entropy under bounded orbit equivalence.
\end{abstract}

\maketitle

\vspace*{-3mm}

\tableofcontents

\vspace*{-3mm}

\section{Introduction}

In the study of von Neumann algebras and their connections to group theory, 
a pivotal role is played by the II$_1$ factors,
i.e., the infinite-dimensional von Neumann algebras having trivial center and a 
(necessarily faithful and unique) normal tracial state.
A II$_1$ factor is said to have {\it property Gamma} if it has an asymptotically central net
of unitaries with trace zero.
This concept was introduced by Murray and von Neumann
in order to show that there exists a II$_1$ factor, in this case the group von Neumann algebra $\sL F_2$ of the
free group on two generators, which is not isomorphic to the hyperfinite II$_1$ factor \cite{MurvNe43}.
While the stronger property of injectivity was later identified as the von Neumann algebra
analogue of amenability for groups,
and indeed as being equivalent to this amenability in the case of group von Neumann algebras,
Murray and von Neumann's argument nevertheless hinted at a connection to amenability
through its use of almost invariant measures and paradoxicality in a way that, as they pointed out,
paralleled Hausdorff's paradoxical decomposition of the sphere.
In fact property Gamma has come, perhaps surprisingly, to play a significant role in the study of
operator-algebraic amenability. It appears as a technical tool in Connes's celebrated work on the classification of
injective factors \cite{Con76}, and its C$^*$-algebraic analogue has very recently
been used to establish the equivalence of finite nuclear and $\cZ$-stability in the
Toms--Winter conjecture \cite{CasEviTikWhiWin18}. The explicit connection to group-theoretic amenability
was made by Effros \cite{Eff75}, who showed that, for an ICC discrete group $G$, if the group von Neumann algebra $\sL G$
(which in this case is a II$_1$ factor)
has property Gamma then $G$ is {\it inner amenable}, meaning that
there exists an atomless finitely additive probability measure on $G\setminus \{ 1_G \}$ which is
invariant under the action of $G$ by conjugation.
The converse of this implication turns out to be false, however, as was shown by Vaes in \cite{Vae12}.

Let us say for brevity that an ICC discrete group $G$ has {\it property Gamma} if $\sL G$ has property Gamma.
When such a group is nonamenable it exhibits the kind of behavior that brings it under the compass
of Popa's deformation/rigidity program.
Indeed Peterson and Sinclair showed in \cite{PetSin12} that if $G$ is countable and nonamenable and has property Gamma then
each of its weakly mixing Gaussian actions, including
its Bernoulli action over a standard atomless base, is $\sU_{\rm fin}$-cocycle superrigid, i.e., every cocycle
taking values in a Polish group that embeds as a closed subgroup of the unitary group of some II$_1$ factor,
and in particular every cocycle into a countable discrete group, is cohomologous
to a homomorphism. This implies that if $G$ has no nontrivial finite normal subgroups
then each of its weakly mixing Gaussian actions is orbit equivalence superrigid, i.e.,
any orbit equivalence with another ergodic p.m.p.\ action of a countable discrete group
implies that the actions are conjugate modulo an isomorphism of the groups.
In \cite{TD14} Ioana and the second author strengthened
Peterson and Sinclair's result by showing that the conclusion holds more generally if property Gamma
is replaced by inner amenability.
In Corollary~D of \cite{ChiSinUdr16} it was shown that inner amenability implies the vanishing of the
first $\ell^2$-Betti number, which is a necessary condition
for $\sU_{\rm fin}$-cocycle superrigidity (Corollary~1.2 of \cite{PetSin12}) and
hence can also be derived from \cite{TD14}.

Inner amenability is also related to Jones and Schmidt's notion of stability for p.m.p.\ equivalence relations \cite{JS87},
which is an anlogue of the McDuff property for II$_1$ factors.
We say that an ergodic p.m.p.\ equivalence relation is {\it JS-stable} if it is isomorphic to its product
with the unique ergodic hyperfinite p.m.p.\ equivalence relation. A countable group
is said to be {\it JS-stable} if it admits a free ergodic p.m.p.\ action whose orbit equivalence
relation is JS-stable. Just as McDuff implies property Gamma in the II$_1$ factor setting,
JS-stability for a group implies inner amenability (Proposition~4.1 of \cite{JS87}).
However, there exist ICC groups with property Gamma which are not JS-stable, as shown by Kida \cite{Kid12}.
Moreover, while JS-stability for an action implies that the crossed product is McDuff,
JS-stability for a group does not imply that the group has property Gamma \cite{Kid15b}.
We refer the reader to \cite{DepVae17} for a more extensive discussion of the relationships
between all of these properties.

In this paper we show that actions of amenable groups on the Cantor set
provide a rich source of groups which have property Gamma and are JS-stable,
and that many of these groups
are simple, finitely generated, and C$^*$-simple (C$^*$-simplicity means that the reduced group C$^*$-algebra
is simple, or equivalently that there are no nontrivial amenable uniformly recurrent subgroups \cite{Ken15},
and is a strengthening of nonamenability).

As mentioned above, inner amenability implies the vanishing of the first $\ell^2$-Betti number.
It is also known that JS-stability implies the vanishing of all $\ell^2$-Betti numbers,
for in this case the group is measure equivalent to a product of the form $H\times \Z$ for some group $H$,
which has the property that all of its $\ell^2$-Betti numbers vanish \cite{CheGro86},
and the vanishing of all $\ell^2$-Betti numbers is an invariant of measure equivalence \cite{Gab02}.
We will begin by proving in Theorem~\ref{T-Betti} that for every action
$\Gamma\curvearrowright X$ of a countably infinite group (amenable or not) on the Cantor set,
every infinite subgroup of the topological full group
containing the alternating group has the property that all of its $\ell^2$-Betti numbers vanish.
We then show in Theorem~\ref{T-JS-stable} that, for an action $\Gamma\curvearrowright X$
of a countable amenable group on the Cantor set, every such subgroup of the topological full group is JS-stable.
To carry this out we use the stability sequence criterion for JS-stability due to Jones and Schmidt
in the free ergodic case \cite{JS87}
and to Kida more generally \cite{Kid15a} along with an invariant random subgroup generalization of a result from \cite{TD14}
that yields a stability sequence from the existence of fiberwise stability sequences
in an equivariant measure disintegration of an action whose orbit equivalence relation is hyperfinite
(Theorem~\ref{thm:IRSstable}).

In Theorem~\ref{T-Gamma} we establish property Gamma for the topological full group, and every subgroup thereof containing
the alternating group $\fA (\Gamma ,X)$,
of a topologically free action $\Gamma\curvearrowright X$
of a countable amenable group on the Cantor set (the ICC condition is automatic in this case by Proposition~\ref{P-ICC}).
As a corollary these subgroups of the topological full group are inner amenable,
and in Section~\ref{S-inner amenability} we provide two direct
proofs of this inner amenability whose techniques may be of use in other contexts.
One of these proofs does not require topological freeness
and yields inner amenability merely assuming that the subgroup is infinite, which is automatic if
$\fA (\Gamma ,X)$ is nontrivial.

Perhaps most interesting is the situation when the action $\Gamma\curvearrowright X$ is minimal and expansive,
as the first of these two conditions implies that the alternating group is simple and the second implies,
under the assumption that $\Gamma$ is finitely generated,
that the alternating group finitely generated \cite{Nek17}.
Under these additional hypotheses it is possible
for the alternating group not to be amenable, as Elek and Monod
constructed in \cite{EM13} a free minimal expansive $\Z^2$-action whose topological full group
contains a copy of $F_2$, and in this case the alternating group coincides
with the commutator subgroup of the topological full group and hence is nonamenable.
We thereby obtain an example of a
simple finitely generated nonamenable group with property Gamma, and in particular an example of
a simple finitely generated nonamenable inner amenable group, which answers a question that was
posed to the second author by Olshanskii.
It has also recently been shown by Sz\H{o}ke that every infinite finitely generated group which is not virtually cyclic
admits a free minimal expansive action on the Cantor set whose topological full group contains $F_2$ \cite{Szo18}.

We push this line of inquiry further by showing that many finitely generated torsion-free amenable groups,
including those that are residually finite and possess a nontorsion element with infinite conjugacy class,
admit an uncountable family of topologically free minimal actions on the Cantor set
whose alternating groups are C$^*$-simple (and in particular nonamenable)
and pairwise nonisomorphic (Theorem~\ref{T-ent nonamenable}).
These actions are constructed so as to have different values of topological entropy,
so that the pairwise nonisomorphism follows by combining the following two facts:
\begin{enumerate}
\item the alternating groups of two minimal actions are isomorphic if and only if the actions are continuously orbit equivalent,

\item for topologically free actions of finitely generated amenable groups which are not virtually cyclic,
topological entropy is an invariant of continuous orbit equivalence.
\end{enumerate}
The first of these follows from a reconstruction theorem of Rubin \cite{Rub89} and an observation of X. Li \cite{Li18}
(see Theorem~\ref{T-alt coe}).
The second is a topological version of a theorem of Austin on measure entropy and bounded orbit equivalence \cite{Aus16}
and follows from Austin's result and the variational principle in the torsion-free case, which is our domain of application
in Theorem~\ref{T-ent nonamenable}. We include in Section~\ref{S-entropy} a demonstration of how (ii) can be directly derived
from the geometric ideas underlying Austin's approach in the measure setting.
The C$^*$-simplicity
is verified by applying a rigid stabilizer criterion of Le Boudec and Matte Bon \cite{LeBMat18} for groups of homeomorphisms
which relies on Kennedy's characterization of C$^*$-simplicity as the absence of nontrivial amenable
uniformly recurrent subgroups \cite{Ken15}. Le Boudec and Matte Bon applied their criterion to show, among other things,
that the full group of a free minimal action of a nonamenable group on the Cantor set is C$^*$-simple.
The novelty in our case is that we are working within the regime of amenable acting groups.
The construction itself relies on the kind of recursive blocking technique that was used in \cite{LinWei00}
to construct minimal $\Z$-actions with prescribed mean dimension.

It should be noted that, for actions of countable amenable groups
on the Cantor set, the topological full group and its alternating subgroup are
often amenable. A celebrated theorem of Juschenko and Monod says that
this is the case for all minimal $\Z$-actions on the Cantor set \cite{JusMon12}.
It is also the case if the action is free, minimal, and equicontinuous (Theorem~1.3 of \cite{CorMed16})
or if the action is free and minimal and the acting group is virtually cyclic \cite{Szo18}.
On the other hand,
the alternating group of the shift action $\Z \curvearrowright \{ 1,\dots ,q \}^{\Z}$ for $q\geq 4$
contains a copy of $F_2$. In fact the strategy of the construction in the proof of
Theorem~\ref{T-nonamenable} involves the embedding of such a shift along some nontorsion direction in the acting group,
which can be done in a way that ensures minimality as long as the group is not virtually cyclic.
Using an idea from \cite{vDo90} as in the construction of Elek and Monod
(see also the discussion in Section~3.7 of \cite{Cor13}),
one then exploits the shift structure to show that the free product $\Z_2 * \Z_2 * \Z_2$,
which contains a copy of $F_2$, embeds into the rigid stabilizer of every nonempty open set
under the action of the alternating group.

We mention finally the Thompson group $V$ as an example of
a group of homeomorphisms of the Cantor set which is C$^*$-simple \cite{LeBMat18}
and not inner amenable \cite{HaaOle17}.
\medskip

\noindent{\it Acknowledgements.}
The first author was partially supported by NSF grant DMS-1500593 and was 
affiliated with Texas A\&M University during the course of the project.
A portion of the work was carried out during his six-month stay in 2017-2018 at the ENS de Lyon,
during which time he held ENS and CNRS visiting professorships.
He thanks Damien Gaboriau and Mikael de la Salle at the ENS for their generous hospitality. The second author was partially supported by NSF grant DMS-1600904.
We thank the referee for helpful suggestions.

\section{Dynamical alternating groups and continuous orbit equivalence}

Throughout the paper group actions on compact spaces will always be understood to be continuous
and the acting groups discrete.
Let $\Gamma$ be a discrete group and $\Gamma\curvearrowright X$ an
action on a compact metrizable space. We will express the action using the concatenation
$\Gamma\times X\ni (s,x)\mapsto sx$, and in the occasional case that we need to notationally distinguish
the action from others by giving it a name $\alpha$ we may also write $(s,x)\mapsto \alpha_s x$.
The action is said to be {\it topologically free} if the set of points in $X$ with trivial 
stabilizer is dense (note that this set is a $G_\delta$ when $\Gamma$ is countable).
Two actions $\Gamma\curvearrowright X$ and $\Gamma\curvearrowright Y$ of $\Gamma$
on compact metrizable spaces are {\it topologically conjugate} if there exists
a homeomorphism $\varphi :X\to Y$ such that $\varphi (sx)=s\varphi (x)$ for all
$s\in \Gamma$ and $x\in X$. Such a homeomorphism is called a {\it topological conjugacy}.

The {\it topological full group} $[[ \Gamma\curvearrowright X ]]$
of the action is the group of all homeomorphisms $T:X\to X$ with the property that
there exists a clopen partition $\{ P_1 , \dots , P_n \}$ of $X$ and $s_1 , \dots , s_n \in \Gamma$ such that $Tx = s_i x$
for every $i=1,\dots ,n$ and $x\in P_i$.
If the action has a name $\alpha$ then we also write $[[\alpha ]]$.
This group is countable since $\Gamma$ is countable
and there are at most countably many clopen partitions of $X$.

Let $d\in\N$.
Consider the collection of all group homomorphisms $\varphi$ from the symmetric group $\fS_d$ on $\{ 1,\dots , d \}$
to $[[ \Gamma\curvearrowright X ]]$ with the property that there are pairwise disjoint clopen sets
$A_1 ,\dots , A_d \subseteq X$ such that for every $\sigma\in\fS_d$ the homeomorphism
$\varphi (\sigma )$ acts as the identity off of $A_1 \sqcup\cdots\sqcup A_d$ and for each
$i=1,\dots ,d$ acts on $A_i$ as some element of $\Gamma$ with image $A_{\sigma (i)}$. Define
$\fS_d (\Gamma ,X)$ to be the subgroup of $[[ \Gamma\curvearrowright X ]]$ generated by the images of
all of these homomorphisms,
and $\fA_d (\Gamma ,X)$ to be the subgroup of $[[ \Gamma\curvearrowright X ]]$ generated by the images of
the alternating group $\fA_d \subseteq \fS_d$ under all of these homomorphisms.
One can also generate these groups by considering more generally the embeddings
of $\fS_d$ as above but with $\varphi (\sigma )$ acting on $A_i$ as an element of $[[ \Gamma\curvearrowright X ]]$
with image $A_{\sigma (i)}$, as it is readily verified that the image of such an embedding lies in a product of embeddings
of $S_d$ of the more restrictive type using single elements of $\Gamma$.

\begin{definition}[\cite{Nek17}]\label{D-alternating}
The {\it symmetric group} $\fS (\Gamma ,X)$ and {\it alternating group} $\fA (\Gamma ,X)$ of the action
$\Gamma\curvearrowright X$ are defined to be $\fS_2 (\Gamma ,X)$ and $\fA_3 (\Gamma ,X)$, respectively.
If the action has a name $\alpha$ then we also write $\fS (\alpha )$ and $\fA (\alpha )$.
\end{definition}

If the action has no finite orbits then $\fS (\Gamma ,X) = \fS_d (\Gamma ,X)$ for every $d\geq 2$
and $\fA (\Gamma ,X) = \fA_d (\Gamma ,X)$ for every $d\geq 3$ (Corollary~3.7 of \cite{Nek17}).

The alternating group $\fA (\Gamma ,X)$ is contained in the commutator subgroup $[[ \Gamma\curvearrowright X ]]'$
of $[[ \Gamma\curvearrowright X ]]$ and is equal to it when the action is almost finite, i.e.,
when $X$ can be partitioned into clopen towers whose shapes are approximately left invariant to within an
arbitrarily prescribed tolerance \cite{Mat12,Nek17} (see the discussion in Section~4 of \cite{Nek17}).

Recall that the action $\Gamma\curvearrowright X$ is {\it expansive} if, fixing a compatible metric $d$ on $X$,
there is an $\eps > 0$ such that for all distinct $x,y\in X$ there is an $s\in\Gamma$ for which $d(sx,sy) \geq\eps$.
When $X$ is zero-dimensional, expansivity
is equivalent to the action being topologically conjugate to a right subshift action over a finite alphabet,
i.e., the restriction of the right shift action $\Gamma\curvearrowright \{ 1,\dots , n\}^\Gamma$ for some $n\in\N$,
given by $(sx)_t = x_{ts}$, to a closed $\Gamma$-invariant subset.
The following facts were established by Nekrashevych in Theorem~4.1, Proposition~5.5,
and Theorem~5.6 of \cite{Nek17}.

\begin{theorem}[\cite{Nek17}]\leavevmode\label{T-simple fg}
\begin{enumerate}
\item If the action $\Gamma\curvearrowright X$ is minimal then $\fA (\Gamma , X)$ is simple.

\item If $\Gamma$ is finitely generated and the action $\Gamma\curvearrowright X$ is expansive
and has no orbits of cardinality less than $5$
then $\fA (\Gamma ,X)$ is finitely generated.
\end{enumerate}
\end{theorem}

Next we define continuous orbit equivalence and record its relationship to the symmetric and alternating groups.

\begin{definition}\label{D-coe}
Two actions $\Gamma\stackrel{\alpha}{\curvearrowright} X$ and $\Lambda\stackrel{\beta}{\curvearrowright} Y$
on compact metrizable spaces are {\it continuously orbit equivalent}
if there are a homeomorphism $\Phi : X\to Y$ and
continuous maps $\kappa : \Gamma\times X\to \Lambda$ and $\lambda : \Lambda\times Y\to \Gamma$ such that
\begin{align*}
\Phi (\alpha_t x) &= \beta_{\kappa (t,x)}\Phi (x) , \\
\Phi^{-1} (\beta_h y) &= \alpha_{\lambda (h,y)} \Phi^{-1} (y)
\end{align*}
for all $t\in \Gamma$, $x\in X$, $h\in \Lambda$, and $y\in Y$.
\end{definition}

It is clear that topological conjugacy
implies continuous orbit equivalence.
If the action $\beta$ is topologically free, then the map $\kappa$
is uniquely determined by the first line of the above display
since $\Phi$ is continuous, and one can moreover verify
the cocycle identity
\begin{align*}
\kappa (st , x) = \kappa (s ,\alpha_t x) \kappa (t , x)
\end{align*}
for all $s,t \in \Gamma$ and $x\in X$.
If $\alpha$ and $\beta$ are both topologically free, then we have
\begin{align*}
\lambda (\kappa (t,x),\Phi (x)) = t
\end{align*}
for all $t\in \Gamma$ and $x\in X$, and $\lambda$ is uniquely determined by this identity.

As pointed out in Section~3.3 of \cite{Nek17}, the equivalences
(i)$\Leftrightarrow$(ii)$\Leftrightarrow$(iii)$\Leftrightarrow$(iv)
in the following theorem are a consequence of a reconstruction theorem of Rubin \cite{Rub89}.
The equivalence (i)$\Leftrightarrow$(ii) was also established by Matui in Theorem~3.10 of \cite{Mat15}.
The equivalence (i)$\Leftrightarrow$(v) was observed by X. Li in Theorem~1.2 of \cite{Li18}.

\begin{theorem}\label{T-alt coe}
Let $\Gamma \curvearrowright X$ and $\Lambda \curvearrowright Y$ be topologically free
minimal actions of countable groups on compact metrizable spaces.
Then the following are equivalent:
\begin{enumerate}
\item $X \rtimes \Gamma$ and $Y \rtimes \Lambda$ are isomorphic as topological groupoids,

\item the groups $[[\Gamma \curvearrowright X ]]$ and $[[\Lambda \curvearrowright Y ]]$ are isomorphic,

\item the groups $\fA (\Gamma , X )$ and $\fA (\Lambda , Y)$ are isomorphic,

\item the groups $\fS (\Gamma , X)$ and $\fS (\Lambda , Y)$ are isomorphic,

\item the actions $\Gamma \curvearrowright X$ and $\Lambda \curvearrowright Y$ are continuously
orbit equivalent.
\end{enumerate}

\end{theorem}

\section{$\ell^2$-Betti numbers}\label{S-Betti}

A sequence $(c_n)_{n\in \N}$ of elements of a group $H$ is said to be {\it asymptotically central} in $H$ if each $h\in H$ commutes with all but finitely many of the $c_n$, and {\it nontrivial} if $\{ c_n \, : \, n\in \N \}$ is infinite.
The existence of a nontrivial asymptotically central sequence in $H$ is equivalent to the centralizer $C_H(F)$ of each finite set $F\subseteq H$ being infinite. If there exist two asymptotically central sequences $(c_n)_{n\in \N}$ and $(d_n)_{n\in \N}$ in $H$ which are noncommuting, i.e., which satisfy $c_nd_n\neq d_nc_n$ for all $n\in \N$,
then it is easy to see that both of these sequences must be nontrivial. The existence of two noncommuting asymptotically central sequences in $H$ is equivalent to $C_H(F)$ being non-Abelian for every finite subset $F$ of $H$.

Now let $\Gamma \curvearrowright X$ be an action of a countable group on the Cantor set.
Let $G$ be any subgroup of $[[ \Gamma \curvearrowright X ]]$ containing $\fA (\Gamma ,X)$.
For a finite set $F\subseteq X$ we write $G_{(F)}$ for the set of all elements in $G$
which fix a neighborhood of every $x\in F$. When $F$ is a singleton $\{ x \}$ we simply write $G_{(x)}$.

\begin{lemma}\label{lem:IAfullgroup}
Let $\Gamma \curvearrowright X$ be an action of a countable group on the Cantor set.
\begin{enumerate}
\item If there is a uniform finite bound on the size of $\Gamma$-orbits, then the group $[[\Gamma \curvearrowright X ]]$ is locally finite.
\item Let $X_0$ denote the set of points $x\in X$ with the property that for every open neighborhood $U$ of $x$ and for every $n\in \N$ there exists some $y\in X$ such that $|\Gamma y \cap U|\geq n$. The set $X_0$ is closed and $\Gamma$-invariant, and $X_0 = \emptyset$ if and only if there is a uniform finite bound on the size of $\Gamma$-orbits.
\item Suppose that $X_0\neq \emptyset$. Let $x\in X_0$ and let $F$ be a finite subset of $X$ containing $x$. Let $G$ be any subgroup of $[[ \Gamma \curvearrowright X ]]$ containing $\fA (\Gamma ,X)$. Then $G_{(x)}$ contains two noncommuting asymptotically central sequences, and these sequences can moreover be taken to lie in $G_{(F)}\cap \fA (\Gamma , X )$.
\end{enumerate}
\end{lemma}

\begin{proof}
Suppose that every $\Gamma$-orbit has size at most $m$. If $H$ is a finitely generated subgroup of $[[\Gamma \curvearrowright X ]]$, then the intersection $N$ of all subgroups of $H$ which have index at most $m$ in $H$ is a finite index normal subgroup of $H$ which acts trivially on $X$, so $N$ is the trivial group and hence $H$ is finite. This proves (i). For (ii), it is clear that $X_0$ is closed and $\Gamma$-invariant, and that if $X_0\neq \emptyset$ then there cannot be a uniform finite bound on the size of $\Gamma$-orbits. Conversely, if $X_0=\emptyset$ then by compactness we can find some $n>0$ and some cover of $X$ by finitely many open sets $V_0,\dots , V_{m-1}$ with the property that $|\Gamma y \cap V_i|<n$ for all $i<m$ and $y\in Y$, and hence every $\Gamma$-orbit has cardinality at most $mn$. This proves (ii). For (iii), since $x\in X_0$ we may find a properly decreasing sequence $U_0\supseteq U_1\supseteq \cdots$ of clopen neighborhoods of $x$ disjoint from $F\setminus \{ x \}$ and with $\bigcap _n U_n = \{ x \}$ such that for each $n\in \N$ the set $U_n\setminus U_{n+1}$ contains at least four elements of some $\Gamma$-orbit. The subgroup $H_n$ of $G_{(F)}\cap \fA (\Gamma , X )$ consisting of all $g\in G$ whose support is contained in $U_n\setminus U_{n+1}$ is then non-Abelian, so we fix any two noncommuting elements, $c_n$ and $d_n$, of this subgroup. Then $(c_n)_{n\in \N}$ and $(d_n)_{n\in \N}$ are noncommuting asymptotically central sequences for $G_{(x)}$.
\end{proof}

\begin{theorem}\label{T-Betti}
Let $\Gamma \curvearrowright X$ be an action of a countable group $\Gamma$ on the Cantor set and suppose that $G$ is an infinite subgroup of $[[\Gamma \curvearrowright X ]]$ containing $\fA (\Gamma ,X )$. Then all $\ell^2$-Betti numbers of $G$ vanish.
\end{theorem}

\begin{proof}
If there is a uniform finite bound on the size of $\Gamma$-orbits then $G$ is locally finite by (i) of Lemma \ref{lem:IAfullgroup}, so all $\ell ^2$-Betti numbers of $G$ vanish. Assume now that there is no such uniform bound on the size of $\Gamma$-orbits. Then the set $X_0$ from (ii) of Lemma \ref{lem:IAfullgroup} is nonempty. We first show that for each finite subset $F$ of $X$ which meets $X_0$, the $\ell^2$-Betti numbers of $G_{(F)}$ all vanish.
By (iii) of Lemma \ref{lem:IAfullgroup}, the group $G_{(F)}$ has a nontrivial asymptotically central sequence $(c_n)_{n\in \N}$.  After moving to a subsequence we may assume additionally that the $c_n$ for $n\in \N$ pairwise commute, and that
the group they generate, which we will call $A$, is an infinite Abelian subgroup of $G_{(F)}$. Note that, for any finite subset $Q$ of $G_{(F)}$, the intersection $\bigcap _{g\in Q} gAg^{-1}$ is infinite since it contains all but finitely many of the $c_n$. It therefore follows from
Corollary~1.5 of \cite{BFS14} that all $\ell^2$-Betti numbers of $G_{(F)}$ vanish.

Now fix any $x\in X_0$. Then for any finite subset $Q$ of $G$, letting $F_Q:= Qx \subseteq X_0$ we have $\bigcap _{g\in Q}gG_{(x)}g^{-1} = G_{(F_Q)}$, and so all $\ell^2$-Betti numbers of this intersection vanish. We can therefore apply Theorem~1.3 of \cite{BFS14} to conclude that all $\ell^2$-Betti numbers of $G$ vanish.
\end{proof}

\section{Jones--Schmidt stability}\label{S-JS-stability}

We continue to use the notation and terminology from the last section.
In what follows $\mc{R}_0$ denotes the unique ergodic hyperfinite type II$_1$ equivalence relation.

\begin{definition}
An ergodic discrete p.m.p.\ equivalence relation $\mc{R}$ is said to be {\it JS-stable} if $\mc{R}$ is isomorphic to the product $\mc{R}\otimes \mc{R}_0$.
A countable group $G$ is said to be {\it JS-stable} if it admits a free ergodic p.m.p.\ action whose orbit equivalence relation is JS-stable.
\end{definition}

Let $G\curvearrowright (Z,\eta )$ be a p.m.p.\ action of a countable discrete group. We denote by $[G\ltimes (Z,\eta )]$ the full group of the action, i.e., the collection of all measurable maps $T:Z\rightarrow G$ with the property that the transformation $T^0: Z\rightarrow Z$, defined by $T^0(z):= T(z) z$, is an automorphism of the measure space $(Z,\eta )$.

\begin{definition}\label{def:stabseq}
Let $G\curvearrowright (Y,\nu )$ be a p.m.p.\ action of a countable discrete group. A {\it stability sequence} for the action is a sequence $(T_n,A_n)_{n\in \N}$ with $T_n\in [G\ltimes (Y,\nu )]$ and $A_n$ a measurable subset of $Y$
such that
\begin{enumerate}
\item[(1)] $\nu (\{ y\in Y \, : \, T_n (gy) =gT_n(y)g^{-1} \} ) \rightarrow 1$ for all $g\in G$,
\item[(2)] $\nu (T_n^0(B)\triangle B ) \rightarrow 0$ for all measurable $B\subseteq Y$,
\item[(3)] $\nu (gA_n\triangle A_n)\rightarrow 0$ for all $g\in G$,
\item[(4)] $\nu (T_n^0(A_n)\triangle A_n ) \geq \frac12$ for all $n$.
\end{enumerate}
\end{definition}

The following theorem was originally established by Jones and Schmidt in \cite{JS87} under the addition
hypotheses of freeness and ergodicity, which were later removed by Kida \cite{Kid15a}.

\begin{theorem}[\cite{JS87,Kid15a}]
Suppose that $G$ has a p.m.p.\ action $G\curvearrowright (Y,\nu )$ which admits a stability sequence. Then $G$ is JS-stable.
\end{theorem}

The following generalizes Theorem~17 of \cite{TD14}, by replacing the coamenable normal subgroup used in that theorem with an invariant random subgroup which is coamenable in an appropriate sense.

\begin{theorem}\label{thm:IRSstable}
Let $G$ be a countable discrete group and let $\pi :(Y,\nu ) \rightarrow (X,\mu )$ be a $G$-equivariant measure-preserving map between p.m.p.\ actions $G\curvearrowright (Y,\nu )$ and $G\curvearrowright (X,\mu )$. Let $(Y,\nu ) = \int _X (Y_x,\nu _x)\, d\mu$ denote the disintegration of $(Y,\nu )$ via $\pi$, so that $Y_x=\pi ^{-1}(x)$, and let $G_x$ denote the stabilizer subgroup at a point $x\in X$. Suppose that
\begin{itemize}
\item[(a)] the orbit equivalence relation of $G\curvearrowright (X,\mu )$ is $\mu$-hyperfinite, and
\item[(b)] the action $G_x\curvearrowright (Y_x,\nu _x)$ admits a stability sequence for $\mu$-almost every $x\in X$.
\end{itemize}
Then the action $G\curvearrowright (Y,\nu )$ admits a stability sequence, and hence $G$ is JS-stable.
\end{theorem}

\begin{proof}
Let $\mc{S}$ denote the orbit equivalence relation of $G\curvearrowright (X,\mu )$. Since $\pi$ is $G$-equivariant and
measure-preserving, uniqueness of measure disintegrations implies that $\nu _{gx} = g_*\nu _x$ for all $g\in G$ and $\mu$-a.e.\ $x\in X$.  After discarding a null set we may assume that this holds for all $x\in X$, that $\mc{S}$ is hyperfinite, and also that $G_x\curvearrowright (Y_x,\nu _x )$ admits a stability sequence for all $x\in X$. Let $F_0\subseteq F_1\subseteq \cdots$ be an increasing sequence of finite subsets of $G$ which exhaust $G$. Let $\mathcal{B}_0\subseteq \mathcal{B}_1 \subseteq \cdots$ be an increasing sequence of finite Boolean algebras of Borel subsets of $Y$ whose union $\mathcal{B}:=\bigcup _n \mathcal{B}_n$ generates the Borel $\sigma$-algebra on $Y$. Let $\mathcal{S}_0\subseteq \mc{S}_1\subseteq \cdots$ be an increasing sequence of finite Borel subequivalence relations of $\mc{S}$ with $\mc{S}=\bigcup _n \mc{S}_n$. For each $n$ let $X_n\subseteq X$ be a Borel transversal for $\mc{S}_n$, and for $x\in X$ let $\theta _n(x)$ be the unique element of $[x]_{\mc{S}_n}\cap X_n$.
Let $\varphi _n : X\rightarrow G$ be any measurable map with $\varphi _n (x)\theta _n (x) =x$. Let $\rho _n :G\times X \rightarrow G$ be given by $\rho _n (g, x):=\varphi _n(gx)^{-1}g\varphi _n (x)$, and note that if $(gx,x)\in \mathcal{S}_n$ then $\varphi _n(gx)^{-1}g\varphi _n (x)\theta _n (x) = \theta _n (x)$, and hence $\rho _n (g, x) \in G_{\theta _n (x)}$.

For each $n$ and $x_0\in X_n$, since $G_{x_0}\curvearrowright (Y_{x_0},\nu _{x_0} )$ admits a stability sequence we may find
an $S_{n,x_0}\in [G_{x_0}\ltimes (Y_{x_0},\nu _{x_0} ) ]$ and a measurable subset $D_{n,x_0}$ of $Y_{x_0}$ satisfying:
\begin{enumerate}[itemindent=*, leftmargin=*]
\item[(i)] $\nu _{x_0}(W_{n,x_0}) >1-1/n$, where $W_{n,x_0}$ is the set of all $y\in Y_{x_0}$ such that for all $g\in F_n$ and $x\in X$ with $x,gx \in [x_0]_{\mc{S}_n}$ one has
\[
S_{n,x_0}(\rho _n (g,x)y) = \rho _n (g,x)S_{n,x_0}(y)\rho _n (g,x)^{-1} ,
\]
\item[(ii)] $\nu _{x_0} (S_{n,x_0}^0(\varphi _n (x)^{-1} B)\triangle \varphi _n (x)^{-1}B)< 1/n$ for all $B\in \mathcal{B}_n$
and $x\in [x_0]_{\mc{S}_n}$,
\item[(iii)] $\nu _{x_0}(\rho (g,x)^{-1} D_{n,x_0}\triangle D_{n,x_0} ) <1/n$ for all $g\in F_n$ and $x\in X$ with $x,gx \in [x_0]_{\mc{S}_n}$,
\item[(iv)] $\nu _{x_0}(S_{n,x_0}(D_{n,x_0})\triangle D_{n,x_0}) \geq 1/2$.
\end{enumerate}
By a standard measurable selection argument that requires discarding a null set in $X$,
we may additionally assume without loss of generality that the maps $x_0\mapsto S_{n,x_0}$ and $x_0\mapsto D_{n,x_0}$ are Borel in the sense that the function $y\mapsto S_{n,\theta _n(\pi (y))}(y)$ from $Y$ to $G$ is Borel
and $\{ y\in Y \, : \, y\in D_{n,\theta _n(\pi (y))} \}$ is a Borel subset of $Y$.

For each $n\in \N$ define $T_n :Y\rightarrow G$ and $A_n\subseteq Y$ by
\begin{align*}
T_n (y) &:= \varphi _n(x)S_{n,\theta _n (x)}(\varphi _n (x)^{-1}y )\varphi _n (x)^{-1} \ \text{ for }y\in Y_x, \ x\in X  \\
A_n&:= \{ y\in Y \, : \, \, y\in \varphi _n (x)D_{n,\theta _n (x)}, \text{ where } x=\pi (y) \} .
\end{align*}
We will show that $(T_n,A_n )_{n\in \N}$ is a stability sequence for $G\curvearrowright (Y,\nu )$. First note that the map $T_n^0: y\mapsto T_n(y)y$ is indeed an automorphism of $(Y,\nu )$, since for every $x\in X$ the restriction $T_{n,x}^0$ of $T_n^0$ to $Y_x$ is an automorphism of $(Y_x,\nu _x )$ (here we are using that $\varphi _n(x)_*\nu _{\theta _n (x)} = \nu _{\varphi _n (x)\theta _n (x)}= \nu _{x}$).
It remains to check that properties (1) through (4) of Definition \ref{def:stabseq} hold for $(T_n,A_n)_{n\in \N}$.

(1): Let $g\in G$. Given $x\in X$, for all large enough $n$ we have $g\in F_n$ and $(gx,x)\in \mathcal{S}_n$ so that $\theta _n(gx)=\theta _n (x)$, and hence for all $y\in \varphi _n(x)W_{n,\theta _n (x)}\subseteq Y_x$ we have
\begin{align*}
T_n(gy) &= \varphi _n(gx)S_{n,\theta _n (gx)} (\varphi _n (gx)^{-1}gy )\varphi _n (gx)^{-1} \\
&=\varphi _n(gx)S_{n,\theta _n (x)} (\rho _n (g,x)\varphi _n (x)^{-1}y )\varphi _n (gx)^{-1} \\
&= \varphi _n (gx)\rho _n (g,x)S_{n,\theta _n(x)}(\varphi _n (x)^{-1}y)\rho _n (g,x)^{-1}\varphi _n (gx)^{-1} \\
&= g\varphi _n (x)S_{n,\theta _n(x)}(\varphi _n (x)^{-1}y)\varphi _n (x)^{-1}g^{-1}\\
&= gT_n(y)g^{-1} .
\end{align*}
Since by (i) we have $\nu _x (\varphi _n (x) W_{n,\theta _n(x)}) = \nu _{\theta _n(x)}(W_{n,\theta _n (x)})>1-1/n$, it follows that (1) holds.

(2): It suffices to show (2) for $B\in \mathcal{B}$. Let $n$ be large enough so that $B\in \mathcal{B}_n$. For $x\in X$ we have $T_{n}^0(B \cap Y_x) = \varphi _n(x)S_{n,\theta _n (x)}^0(\varphi _n (x)^{-1}B) \cap Y_x$, and hence by (ii) applied to $x_0=\theta _n (x)$ we have
\[
\nu _x (T_{n}^0(B)\triangle B) = \nu _{\theta _n (x)}(S_{n,\theta _n (x)}^0(\varphi _n (x)^{-1}B)\triangle \varphi _n (x)^{-1}B ) < 1/n .
\]
It follows that $\nu (T_n^0(B)\triangle B) < 1/n$ for all large enough $n$, and hence (2) holds.

(3): Let $g\in G$. Given $x\in X$, for all large enough $n$ we have $g\in F_n$ and $(gx,x)\in \mathcal{S}_n$, and hence for $y\in Y_x$ we have $gy\in A_n$ if and only if $y\in g^{-1}\varphi _n (gx)D_{n,\theta _n (gx)}= \varphi _n(x)\rho _n(g,x)^{-1}D_{\theta _n(x)}$. Therefore
\begin{align*}
\nu _x (g^{-1} A_n \triangle A_n )
&= \nu _x(\varphi _n (x)\rho _n (g,x)^{-1}D_{\theta _n (x)}\triangle \varphi _n (x)D_{\theta _n (x)} ) \\
& = \nu _{\theta _n(x)} (\rho _n (g,x)^{-1}D_{\theta _n(x)}\triangle D_{\theta _n (x)} ) < 1/n
\end{align*}
by (iii). Since $\nu (g^{-1}A_n\triangle A_n )= \int _X \nu _x (g^{-1} A_n \triangle A_n ) \, d\mu (x)$, property (3) follows.

(4): For each $x\in X$, we have
\[
(T_n^0(A_n)\triangle A_n )\cap Y_x = \varphi _n (x)(S^0_{n,\theta _n(x)}(D_{n,\theta _n(x)})\triangle D_{n,\theta _n(x)})
\]
and hence $\nu _x (T_n^0(A_n)\triangle A_n ) = \nu _{\theta _n (x)}(S^0_{n,\theta _n(x)}(D_{n,\theta _n(x)})\triangle D_{n,\theta _n(x)}) \geq 1/2$ by (iv). This implies (4).
\end{proof}

In the remainder of this section we write $\lambda$ for the Lebesgue measure on $[0,1]$.

The next lemma is a continuous base version of Proposition~9.8 in \cite{Ke10}.

\begin{lemma}[\cite{Ke10}]\label{lem:Kec}
Let $H$ be a countable discrete group, let $H^*:= H\setminus \{ 1_H \}$, and consider the generalized Bernoulli action $H \curvearrowright ([0,1]^{H^*}, \lambda ^{H^*})$, associated to the conjugation action $H\curvearrowright H^*$.  Suppose that there exist sequences $(c_n)_{n\in \N}$, $(d_n)_{n\in \N}$ in $H$ which are asymptotically central in $H$ and satisfy $c_nd_n\neq d_nc_n$ for all $n$. Then the action $H\curvearrowright ([0,1]^{H^*},\lambda ^{H^*} )$ admits a stability sequence.
\end{lemma}

\begin{proof}
Define $S_n : [0,1]^{H^*} \rightarrow H$ to be the constant map $S_n(y):=c_n$, and define $D_n := \{ y\in [0,1]^{H^*} : y(d_n) \in [0,\tfrac{1}{2}] \}$ (note that $d_n\neq 1_G$ since $c_nd_n\neq d_nc_n$, and so this definition makes sense).
Then $(S_n,D_n)$ is a stability sequence for $H\curvearrowright ([0,1]^{H^*},\lambda ^{H^*} )$.
Indeed the first three conditions in Definition~\ref{def:stabseq} follow directly from asymptotic centrality, while
for the fourth we observe that since
$S_n (D_n) = \{ y\in [0,1]^{H^*} : y(c_n d_n c_n^{-1} ) \in [0,\tfrac{1}{2}] \}$
and $c_n d_n c_n^{-1} \neq d_n$
we have $\lambda ^{H^*} (S_n (D_n ) \triangle D_n ) = \lambda ^{H^*} (S_n (D_n ))\lambda ^{H^*} (D_n ) = \frac12$.\end{proof}

In what follows we will make use of the notion of amenability, relative to a measure, of a countable Borel equivalence relation. There are a variety of equivalent ways of defining amenability in this context, and we will find it convenient to use a definition, essentially coming from \cite{Hjo06}, in terms of actions on Borel bundles of compact metric spaces.

\begin{definition}
Let $W$ be a standard Borel space and let $\mathcal{R}$ be a countable Borel equivalence relation on $W$.

A {\it Borel bundle of compact metric spaces} over $W$ is a standard Borel space $M$ together with Borel maps $\pi : M\rightarrow W$ and $d: M\ast M \rightarrow [0,\infty )$ (where $M\ast M$ denotes the set $\{ (a,b)\in M\times M :\pi (a)=\pi (b)\}$) such that for each $w\in W$ the restriction of $d$ to $M_w\times M_w$ is a metric on $M_w=\pi ^{-1}(w)$ which is compact.

Given such a bundle $M$, for $w\in W$ let $P(M_w)$ denote the (compact metrizable) space of all Borel probability measures on $M_w$, and let $P(M)=\bigcup _{w\in W} P(M_w)$. Then $P$ is also naturally a Borel bundle of compact metric spaces over $W$ (see \cite[Appendix C]{Bow18} for details). A Borel {\it section} of $P(M)$ is a Borel map $W\rightarrow P(M)$, $w\mapsto \lambda _w$, with $\lambda _w \in P(M_w)$ for all $w\in W$.

A {\it Borel action of $\mathcal{R}$ by fiberwise homeomorphisms on $M$} is an assignment of a homeomorphism $\alpha (w_1,w_0) : M_{w_0}\rightarrow M_{w_1}$ to each $(w_1,w_0)\in \mathcal{R}$ satisfying the condition $\alpha (w_2,w_1)\alpha (w_1,w_0)=\alpha (w_2,w_0)$ for all $(w_2,w_1),(w_1,w_0)\in \mathcal{R}$ and such that the map $((w_1,w_0), a)\mapsto \alpha (w_1,w_0)a$ is Borel from $\{ ((w_1,w_0),a)\in \mathcal{R}\times M : a\in M_{w_0}\}$ to $M$.

Let $\rho$ be a probability measure on $W$ which is preserved by $\mathcal{R}$, and let $\alpha$ be a Borel action of $\mathcal{R}$ on $M$ as above. A Borel section $w\mapsto \lambda _w$ of $P(M)$ is said to be {\it $\rho$-a.e.\ $(\mathcal{R},\alpha )$-invariant} if there exists a $\rho$-conull subset $W'$ of $W$ such that $\alpha (w_1,w_0)_*\lambda _{w_0}=\lambda _{w_1}$ for all $(w_1,w_0)\in \mathcal{R}|W'$.

We say that the equivalence relation $\mathcal{R}$ is {\it $\rho$-amenable} if, for every Borel action $\alpha$ of $\mc{R}$ by homeomorphisms on a Borel bundle $M$ of compact metric spaces over $W$, there exists a Borel section of $P(M)$ which is $\rho$-a.e.\ $(\mathcal{R},\alpha )$-invariant.
\end{definition}

\begin{lemma}\label{L-amenable OE}
Let $p: (W,\rho )\rightarrow (X,\mu )$ be a $G$-equivariant measure-preserving map between p.m.p.\ actions $G\curvearrowright (W,\rho )$ and $G\curvearrowright (X,\mu )$. Let $(W,\rho ) = \int _X (W_x,\rho _x)\, d\mu$ denote the disintegration of $(W,\rho )$ via $p$, so that $W_x=p ^{-1}(x)$, and let $G_x$ denote the stabilizer subgroup at $x\in X$. Suppose that
\begin{enumerate}
\item[(a)] the action $G\curvearrowright (X,\mu )$ generates a $\mu$-amenable orbit equivalence relation, and
\item[(b)] for $\mu$-a.e.\ $x\in X$ the action of $G_{x}$ on $(W_x,\rho _x )$ generates a $\rho _x$-amenable orbit equivalence relation.
\end{enumerate}
Then the action $G\curvearrowright (W,\rho )$ generates a $\rho$-amenable orbit equivalence relation.
\end{lemma}

\begin{proof}
Let $\mathcal{S}$ and $\mathcal{R}$ be the orbit equivalence relations generated by the actions of $G$ on $X$ and $W$ respectively, and for each $x\in X$ let $\mathcal{R}_x$ be the restriction of $\mathcal{R}$ to $W_x$, so that $\mathcal{R}_x$ is generated by the action of $G_x$ on $W_x$. Let $\alpha : \mathcal{R}\curvearrowright  M$ be a Borel action of $\mathcal{R}$ by fiberwise homeomorphisms
on a Borel bundle of compact metric spaces over $W$. For each $x\in X$, by restricting to $\mathcal{R}_x$ we obtain the action $\alpha _x : \mathcal{R}_x\curvearrowright M_x$, where $M_x := \bigcup _{w\in W_x}M_w$ is a Borel bundle of compact metric spaces over $W_x$. For $w\in W$ let $P_w$ denote the space of probability measures on $M_w$, and for $x\in X$ let $P_x=\bigcup _{w\in W_x}P_w$.

Let $\mathcal{L}_x$ denote the collection of all measurable sections $W_x\rightarrow P_x$, $w\mapsto \lambda  _w\in P_w$ which are $\rho _x$-a.e.\ $(\mc{R}_x, \alpha _x )$-invariant, i.e., which satisfy $\lambda _{w_1}=\alpha _x (w_1,w_0)_*\lambda _{w_0}$ for a.e.\ $(w_1,w_0) \in \mathcal{R}_x$, and where we identify two such sections if they agree $\rho _x$-almost everywhere. The assumption (b) implies that the set $\mathcal{L}_x$ is nonempty for a.e.\ $x\in  X$. The set $\mathcal{L}_x$ is naturally a compact metrizable space since it can be identified with a weak$^*$-compact subset of the dual of the separable Banach space $L^1 ((W_x,\rho _x),(C(M_w ))_{w\in W_x})$ of all measurable assignments $w\mapsto f_w \in C(M_w)$ with $\int _W \| f _w \| _{\infty} \, d\rho _x (w) <\infty$ (where two such assignments are identified if they agree $\rho _x$-almost everywhere). Then $\mathcal{L}:=\bigcup _{x\in X}\mathcal{L}_x$ is naturally a Borel bundle of compact metric spaces over $X$, and we have a natural Borel action $\beta : \mathcal{S}\curvearrowright \mathcal{L}$ by fiberwise homeomorphisms with $\beta (x,gx) : \mathcal{L}_{gx}\rightarrow \mathcal{L}_{x}$ being given by $(\beta (x,gx ) \lambda )_w := \alpha (w,gw)_*\lambda _{gw}$ for $x\in X$, $g\in G$, and $w\in W_{x}$. This action is well defined, for if $g_0x = g_1x$ then $g_1g_0^{-1} \in G_{g_0x}$ so that $\alpha (g_1w,g_0w )_*\lambda _{g_0w} = \lambda _{g_1w}$ and hence $\alpha (w,g_1w)_*\lambda _{g_1w} = \alpha (w,g_0w)_*\lambda _{g_0w}$.

Since $\mathcal{S}$ is $\mu$-amenable, there exists a measurable section $x\mapsto \omega _x \in P(\mathcal{L}_x)$ which is $\mu$-a.e.\ $(\mathcal{S}, \beta )$-invariant.
Let $\nu _x\in \mathcal{L}_x$ denote the barycenter of $\omega _x$. Then $x\mapsto \nu _x$ satisfies $\beta (gx,x)\nu _x = \nu _{gx}$ for $\mu$-a.e.\ $x\in X$ and all $g$ with $(gx,x)\in \mathcal{S}$, and hence the map $w\mapsto (\nu _{p(w)} )_w$ is $\rho$-a.e.\ $(\mathcal{R},\alpha )$-invariant. This shows that $\mathcal{R}$ is $\rho$-amenable.
\end{proof}

\begin{theorem}\label{T-JS-stable}
Let $\Gamma \curvearrowright X$ be an action of a countable amenable group
on the Cantor set. Let $G$ be an infinite subgroup of $[[ \Gamma\curvearrowright X ]]$ which contains $\fA (\Gamma ,X)$.
Then $G$ is JS-stable.
\end{theorem}

\begin{proof}
If there is a uniform finite bound on the size of $\Gamma$-orbits then $G$ is locally finite by (i) of Lemma \ref{lem:IAfullgroup}, hence JS-stable. We may therefore assume that no such bound exists and hence, by (ii) of Lemma \ref{lem:IAfullgroup}, that the set $X_0$ from Lemma \ref{lem:IAfullgroup} is nonempty. Since $X_0$ is closed and $\Gamma$-invariant, and $\Gamma$ is amenable, there exists a $\Gamma$-invariant Borel probability measure $\mu$ on $X$ with $\mu (X_0)=1$. The measure $\mu$ is then invariant under the group $G$ as well. The orbit equivalence relation of the action $G\curvearrowright (X,\mu )$ is exactly the orbit equivalence relation of $\Gamma$, and hence is $\mu$-hyperfinite by \cite{OrnWei87}. We break the remainder of the proof into two cases.

Case 1: $G_x=G_{(x)}$ for each $x\in X$. By (iii) of Lemma~\ref{lem:IAfullgroup}, for $\mu$-almost every $x\in X$ (namely, every $x\in X_0$) the group $H=G_{x}$ satisfies the hypothesis of Lemma \ref{lem:Kec}. For each $x\in X$ let $(Y_x,\nu _x) = ([0,1]^{G_x^*}, \lambda ^{G_x^*})$. Let $Y := \bigsqcup _{x\in X}Y_x$, let $\pi :Y\rightarrow X$ denote the projection, and equip $Y$ with the $\sigma$-algebra generated by $\pi$ along with the maps $y\mapsto y(g)$ from $\pi ^{-1}(\{ x\, : \, g\in G_x \} )$ to $[0,1]$ for $g\in G$, which is easily seen to be standard Borel (since $Y$ can be identified with a Borel subset of $X\times ([0,1] \cup \{ 2 \} )^{G^*}$). Then $\nu := \int _X \nu _x \, d\mu$ defines a Borel probability measure on $Y$. We have a natural action $G\curvearrowright (Y,\nu )$: for $g\in G$ and $y\in Y_x$, we define $gy\in Y_{gx}$ by $(gy)(h) := y (g^{-1}hg)$. Observe that $g_*\nu _x = \nu _{gx}$, so this action preserves the measure $\nu$, and it makes $\pi$ a $G$-equivariant map, with $x\mapsto \nu _x$ the associated disintegration. In addition, for each $x\in X_0$, this action restricted to $G_x\curvearrowright (Y_x,\nu _x )$ coincides with the generalized Bernoulli action associated to conjugation $G_x\curvearrowright G_x^*$, and so by Lemma \ref{lem:Kec} it admits a stability sequence. Therefore, by Theorem \ref{thm:IRSstable} the action $G\curvearrowright (Y,\nu )$ admits a stability sequence and hence $G$ is JS-stable.

Case 2: The general case. Let $\mu _0$ denote the measure on the space of all subgroups of $G$ obtained as the pushforward of the measure $\mu$ under the map $x\mapsto G_{(x)}$. By \cite{AGV12, CP12}, we may find a p.m.p.\ action $G\curvearrowright (W,\rho )$ with the property that $\mu _0$ is the pushforward of the measure $\rho$ under the stabilizer map $w \mapsto G_w$. We let $G\curvearrowright (\tilde{X},\tilde{\mu}) := (X,\mu )\otimes _{\mu _0} (W,\rho )$ denote the relatively independent joining of the actions on $X$ and $W$. Thus $\tilde{X} =\{ (x,w)\in X\times W \, : \, G_{(x)}= G_w \}$, and the stabilizer of $(x,w)\in \tilde{X}$ is $G_{(x,w)}=G_x\cap G_w= G_x\cap G_{(x)}=G_{(x)}$. Let $\tilde{\mc{S}}$ denote the orbit equivalence relation of the action of $G$ on $\tilde{X}$. Since the orbit equivalence relation of the action $G\curvearrowright (X,\mu )$ is $\mu$-amenable
by the amenability of $\Gamma$ and the groups $G_x/G_{(x)}$ for $x\in X$ are amenable, we deduce
that $\tilde{\mc{S}}$ is $\tilde{\mu}$-amenable by Lemma~\ref{L-amenable OE} and consequently $\tilde{\mu}$-hyperfinite
by the Connes--Feldman--Weiss theorem \cite{CFW81}. The action of $G$ on $(\tilde{X},\tilde{\mu})$ thus
satisfies all of the properties which allow us to proceed exactly as in Case 1
and conclude that $G$ is JS-stable.
\end{proof}

\section{Property Gamma}\label{S-Gamma}

A II$_1$ factor $M$ with trace $\tau$ is said to have {\it property Gamma}
if for every $\eps > 0$ and finite set $\Omega\subseteq M$
there exists a unitary $u\in M$ with $\tau (u) = 0$ such that $\| [a,u] \|_2 < \eps$ for all $a\in\Omega$.
As the proof of Theorem~2.1 in \cite{Con76} shows, this is equivalent to the
existence, for every $\eps > 0$ and finite set $\Omega\subseteq M$,
of a projection $p\in M$ such that $\tau (p) = \frac12$ and $\| [a,p] \|_2 < \eps$ for all $a\in\Omega$.
One may also equivalently require that $|\tau (p) - \frac12 | < \eps$, as one can replace $p$
by a subprojection or superprojection with trace $\frac12$ and adjust $\eps$ accordingly.
We say that an ICC countable discrete group $G$ has {\it property Gamma} if its group von Neumann algebra $\sL G$
(which is a II$_1$ factor in this case by the ICC condition) has property Gamma.

We begin by observing the following.

\begin{proposition}\label{P-ICC}
Let $\Gamma\curvearrowright X$ be an action of a group on the Cantor set. Assume that the set of points whose orbit contains at least four points is dense in $X$. Then every subgroup of $[[\Gamma\curvearrowright X ]]$ containing $\fA (\Gamma, X)$ is ICC.
\end{proposition}

\begin{proof}
Let $U$ denote the set of points whose orbit contains at least four points, so that $U$ is open, and by hypothesis it is dense.  Given a nonidentity $g \in [[\Gamma\curvearrowright X ]]$, it suffices to show that $\{ hgh^{-1} : h\in \fA (\Gamma ,X ) \}$ is infinite. The support $\mathrm{supp}(g)$, of $g$, is a nonempty open set, so we may find an infinite sequence $(F_n)_{n\in \N}$ of pairwise disjoint subsets of $U\cap \mathrm{supp}(g)$ of size $|F_n|=4$ along with $x_n\in F_n$, $n\in \N$, such that $gx_n \in F_n$ and $F_n\subseteq \Gamma x_n$ for every $n\in \N$. This ensures that for each $n\in \N$ we may find some $h_n \in \fA (\Gamma , X)$ satisfying $h_nx_i = x_i$ for all $i\leq n$, $h_ngx_i = gx_i$ for all $i< n$, and $h_ngx_n \neq gx_n$. Then for each $i<n$ the conjugates $h_igh_i^{-1}$ and $h_ngh_n^{-1}$ are distinct since $h_igh_i^{-1}x_i = h_i gx_i \neq gx_i$, whereas $h_ngh_n^{-1}x_i = h_ngx_i = gx_i$.
\end{proof}

Next we establish Lemma~\ref{L-half}, which is a local version of the implication
(i)$\Rightarrow$(ii) of Theorem~1.2 in \cite{KecTsa08} and follows from the same argument,
which we reproduce here. See Section~4 of \cite{KecTsa08} for more details.

Given a set $Y$ and an inclusion of nonempty sets $L\subseteq K$, we write $\pi_L$ for the
coordinate projection map $Y^K \to Y^L$, or simply $\pi_q$ if $L$ is a singleton $\{ q \}$.
We view $\Sym (K)$ as acting on $Y^K$
via left shifts, i.e., $(\sigma y)(t) = y(\sigma^{-1} t)$ for all $y\in Y^K$, $t\in K$,
and $\sigma\in\Sym (K)$, where we use the action notation $t\mapsto\sigma^{-1} t$
for permutations $\sigma$ of $K$.

\begin{lemma}\label{L-CLT}
Let $I = [-1,1]$ and let $\nu$ be a Borel probability measure on $I$ which is centered at $0$
and does not give $0$ full measure.
Let $\eps > 0$. Then there is a $\delta > 0$ such that if $Q = J\sqcup K\sqcup L$ is
a partition of a finite set $Q$ with $|J| \geq (1-\delta )|Q|$ and $|K|=|L|$ then
the three random variables
\begin{align*}
U = \sum_{q\in J} \pi_q ,
\hspace*{4mm}
V = \sum_{q\in K} \pi_q ,
\hspace*{4mm}
W = \sum_{q\in L} \pi_q
\end{align*}
on $I^Q$ with the product measure $\nu^Q$ satisfy
\begin{align*}
\prob (W \leq - U < V) < \eps .
\end{align*}
\end{lemma}

\begin{proof}
Suppose that no such $\delta$ as in the lemma statement exists.
Then for every $n\in\N$ we can find a partition $Q_n = J_n \sqcup K_n \sqcup L_n$ of a finite set $Q_n$
such that $|J_n | \geq (1-1/n)|Q_n |$ and the three random variables
\begin{align*}
U_n = \sum_{q\in J_n} \pi_q ,
\hspace*{4mm}
V_n = \sum_{q\in K_n} \pi_q ,
\hspace*{4mm}
W_n = \sum_{q\in L_n} \pi_q
\end{align*}
on $I^Q$ with the product measure $\nu^Q$ satisfy
\begin{align*}
\prob (W \leq - U < V) \geq \eps .
\end{align*}
The random variables $U_n$, $V_n$, and $W_n$ each have expectation zero. We also note,
writing $\sigma^2 = \int_I x^2 d\nu (x)$, $M_n = |J_n |$, and $m_n = |K_n | = |L_n |$, that
the variance of $U_n$ is equal to $M_n \sigma^2$ while the variance of each of $V_n$ and $W_n$ is
equal to $m_n \sigma^2$.

By passing to a subsequence, we may assume that the sequence $(m_n )$ is either bounded or tends to $\infty$.
If it has a bound $R$ then $|V_n | , |W_n | \leq R$ for all $n$ and so
\begin{align*}
\prob (W_n \leq -U_n < V_n )
&\leq \prob (-R < U \leq R) \\
&= \prob \big(-R/(\sigma \sqrt{M_n} ) < U_n /(\sigma \sqrt{M_n} ) \leq R/(\sigma \sqrt{M_n} )\big) . \notag
\end{align*}
Now $R/(\sigma \sqrt{M_n} ) \to 0$ and $U_n /(\sigma \sqrt{M_n} )$ converges in distribution
(to a standard normal random variable) by the central limit theorem, and so by a simple approximation argument
we see that the last expression in the above display must converge to zero. We may therefore assume
that $m_n \to \infty$.

By the central limit theorem, the random variables $U_n /(\sigma\sqrt{M_n} )$,
$V_n /(\sigma\sqrt{m_n} )$, and $W_n /(\sigma\sqrt{m_n} )$ all converge in distribution
to a standard normal random variable $Z$.
Setting $\lambda_n = \sqrt{m_n /M_n}$ we have $\lambda_n \to 0$ and
so $\lambda_n V_n /(\sigma\sqrt{m_n} )$ and $\lambda_n W_n /(\sigma\sqrt{m_n} )$
converge in probability to zero, as is straightforward to check.
But then, like in the previous case, the quantity
\begin{align*}
\prob (W_n \leq -U_n < V_n )
= \prob \big(\lambda_n W_n /(\sigma\sqrt{m_n} ) \leq -U_n /(\sigma\sqrt{M_n} ) < \lambda_n V_n /(\sigma\sqrt{m_n} )\big)
\end{align*}
must converge to zero. This yields a contradiction, thus establishing the lemma.
\end{proof}

\begin{lemma}\label{L-half}
Let $(Y,\nu )$ be a standard probability space which does give any singleton full measure.
Let $\eps > 0$. Then there exists a $\delta > 0$ such that
for every nonempty finite set $F$ and
every $E\subseteq F$ with $|E|\geq (1-\delta )|F|$ there exists a set $A\subseteq Y^F$
with $A = \pi_E^{-1} (\pi_E (A))$ such that
$|\nu^F (A) - \frac12 | < \eps$ and $\nu^F (\sigma A \triangle A) < \eps$
for all $\sigma\in\Sym (F)$.
\end{lemma}

\begin{proof}
We may assume that $Y = [-1,1]$ and that $\nu$ is centered at $0$.
Let $\delta > 0$ be as given by Lemma~\ref{L-CLT} with respect to $\eps /2$ and let us show that
$\delta /2$ does the required job.
Let $F$ be a nonempty finite set and let $E$ be a subset of $F$ with $|E|\geq (1-\delta /2)|F|$.
Let $\sigma\in\Sym (F)$. Define $A = \{ y\in Y^F : \sum_{s\in E} y_s > 0 \}$.
Then $A = \pi_E^{-1} (\pi_E (A))$.
Define the three random variables
\begin{align*}
U = \sum_{s\in \sigma E \cap E} \pi_s ,
\hspace*{4mm}
V = \sum_{s\in E \setminus \sigma E} \pi_s,
\hspace*{4mm}
W = \sum_{s\in \sigma E \setminus E} \pi_s .
\end{align*}
Since $|\sigma E\cap E| = |F| - |F\setminus \sigma E| - |F\setminus E|
= |F| - 2|F\setminus E| \geq (1-\delta )|F|$, by our choice of $\delta$ via Lemma~\ref{L-CLT} we then have
\begin{align*}
\nu^F (\sigma A \triangle A)
&= \nu^F (A \setminus \sigma A) + \nu^{F_k} (\sigma A \setminus A) \\
&= 2\nu^F (A \setminus \sigma A)
= 2\prob (W \leq - U < V) < \eps . \qedhere
\end{align*}
\end{proof}

For a group $\Gamma$, finite sets $F,T\subseteq\Gamma$, and a $\delta > 0$, we say that $F$ is
{\it $(T,\delta )$-invariant} if $| F \cap \bigcap_{s\in T} s^{-1} F| \geq (1-\delta )|F|$.

\begin{lemma}\label{L-translates}
Let $\Gamma$ be a countably infinite amenable group.
Let $T$ be a finite subset of $\Gamma$ containing $1_\Gamma$ and let $\delta > 0$.
Let $F$ be a $(T,\frac12 )$-invariant nonempty finite subset of $\Gamma$.
Then there exists a $C\subseteq \Gamma$ with $|C| \geq |F|/(2|T|^2)$
such that the sets $Tc$ for $c\in C$ are pairwise disjoint and contained in $F$.
\end{lemma}

\begin{proof}
Write $F' = \bigcap_{s\in T} s^{-1} F$.
Take a maximal set $C\subseteq F'$ with the property that the sets $Tc$ for $c\in C$
are pairwise disjoint. Note that $TC\subseteq F$. We also have $F' \subseteq T^{-1} TC$,
for otherwise taking a $d\in F' \setminus T^{-1} TC$ we would have
$Td \cap Tc = \emptyset$ for all $c\in C$, contradicting maximality.
Finally, since $F$ is $(T,\frac12 )$-invariant we have
$|F| \leq 2|F'| \leq 2|T|^2 |C|$, and so $C$ does the job.
\end{proof}

\begin{theorem}\label{T-Gamma}
Let $\Gamma\curvearrowright X$ be a topologically free action
of a countably infinite amenable group on the Cantor set. Then every subgroup
$[[\Gamma\curvearrowright X ]]$ containing $\fA (\Gamma, X)$ has property Gamma.
\end{theorem}

\begin{proof}
Let $\Omega$ be a finite symmetric subset of $[[\Gamma\curvearrowright X ]]$.
By the definition of the topological full group,
we can find a clopen partition $\{ P_1 , \dots , P_m \}$ of $X$
such that for each $h\in\Omega$ there exist $s_{h,1} , \dots , s_{h,m} \in \Gamma$ for
which $hx = s_{h,i} x$ for every $i=1,\dots , m$ and $x\in P_i$.
Let $K$ be the collection of all of these $s_{h,i}$.

Let $\delta > 0$ be as given by Lemma~\ref{L-half} with respect to $\eps$. By amenability
there exists a $(K , \delta )$-invariant nonempty finite set $T\subseteq\Gamma$ containing $1_\Gamma$.
Again by amenability there is a $(T,\frac12 )$-invariant nonempty
finite set $F\subseteq\Gamma$ which, since $\Gamma$ is infinite, we can take to have cardinality
greater than $6|K^{\Omega\times T}||T|^2$.
Then by Lemma~\ref{L-translates} we can find a $C\subseteq\Gamma$ with
$|C| \geq |F|/(2|T|^2) > 3|K^{\Omega\times T}|$ such that the sets $Tc$ for $c\in C$
are pairwise disjoint and contained in $F$.

By topological freeness there exists an $x_0 \in X$ such that the map $t\mapsto tx_0$
from $F$ to $X$ is injective. By continuity we can then find a clopen neighborhood
$B$ of $x_0$ such that the sets $tB$ for $t\in F$ are pairwise disjoint and
for each $t\in F$ there is an $1\leq i\leq m$ such that $tB \subseteq P_i$.

For every $c\in C$ we have a function $\theta_c \in K^{\Omega\times T}$
such that $htcx = \theta_c (h,t)tcx$ for all $h\in\Omega$, $t\in T$, and $x\in B$. Since $|C| > 3|K^{\Omega\times T}|$,
by the pigeonhole principle there are a $\theta\in K^{\Omega\times T}$ and a $D\subseteq C$
with $|D| = 3$ such that $\theta_c = \theta$ for each $c\in D$.

Set $T' = \bigcap_{s\in K} s^{-1} T$. For each $s\in K$ we have $sT' \subseteq T$
and so for every $h\in\Omega$ we can find a $\sigma_h \in\Sym (T)$ such that $\sigma_h t = \theta (h,t)t$
for all $t\in T'$. Consider the alternating group $\fA (D)$
on the $3$-element set $D$. This is a cyclic group of order $3$.
We regard the product $H = \fA (D)^{T}$ as a subgroup
of $\fA (\Gamma , X )$ with an element $(\omega_t )_{t\in T}$ in $H$ acting by
$tcx \mapsto t\omega_t (c)x$ for all $x\in B$, $t\in T$, and $c\in D$
and by $x\mapsto x$ for all $x\in X\setminus TDB$.

Write $Y$ for the spectrum of the commutative von Neumann algebra $\vN \fA (D)$,
and identify $\vN \fA (D)$ with $\Cb^Y$. We have $|Y| = 3$.
Write $\nu$ for the uniform probability measure on $Y$.
We regard the commutative von Neumann algebra
$\vN H \cong \vN \fA (D)^{\otimes T}$ as $\Cb^{Y^T} \cong (\Cb^Y )^{\otimes T}$
via the canonical identifications.
The canonical tracial state on $\vN H$ is then
given by integration with respect to $\nu^T$.

For $h\in\Omega$ we consider $\sigma_h$ as acting on $Y^T$
by $(\sigma_h x)_t = x_{\sigma_h^{-1} t}$ for all $t\in T$.
According to our invocation of Lemma~\ref{L-half} above,
there exists a set $A\subseteq Y^T$ with $A = \pi_{T'}^{-1} (\pi_{T'} (A))$ such that
$|\nu^T (A) - \frac12 | < \eps$ and $\nu^T (\sigma_h A \triangle A) < \eps$ for all $h\in\Omega$.
Write $p$ for the indicator function of $A$ in $\Cb^{Y^T}$, which we view
as a projection in $\vN \fA (\Gamma , X)$ under the canonical
inclusion of $\vN H \cong\Cb^{Y^T}$ in $\vN \fA (\Gamma , X)$.
Since the canonical tracial state $\tau$ on $\vN \fA (\Gamma , X)$ restricts
to integration with respect to $\nu^T$ on $\vN H \cong\Cb^{Y^T}$,
we have $|\tau (p) - \frac12 | < \eps$.

Now let $h\in\Omega$. Given a
$g = (g_t )_{t\in T}$ in $\fA (D)^T$ with $g_t = 1_{\fA (D)}$ for all $t\in T\setminus T'$
and viewing $\fA (D)^T$ as subgroup of $\fA (\Gamma , X)$,
the element $hgh^{-1}$ also belongs to $\fA (D)^{T}$, with its value at
a given coordinate $t\in T'$ being equal to $g_{\sigma_h^{-1} t}$.
Thus for every elementary tensor $a = \otimes_{t\in T} a_t$ in
$\vN\fA (D)^{\otimes T}$ satisfying $a_t = 1$ for all $t\in T'$ we have
$u_h a u_h^{-1} = \otimes_{t\in T} a_{\sigma_h^{-1} t}$.
Since $A = \pi_{T'}^{-1} (\pi_{T'} (A))$, the projection $p$ is a sum of minimal projections in $\vN\fA (D)^{\otimes T}$
which can be expressed as such elementary tensors and so $u_h p u_h^{-1}$ is equal to
the indicator function of $\sigma_h A$ in $\Cb^{X^T}$ under the identification of
$\Cb^{X^T}$ with $\vN\fA (D)^{\otimes T}$. It follows that
\begin{align*}
\| u_h p u_h^{-1} - p \|_2 = \nu^T (\sigma_s A \triangle A)^{1/2} < \eps^{1/2} .
\end{align*}
We conclude that $\vN\fA (\Gamma , X)$ has property Gamma.
\end{proof}

\begin{remark}
By \cite{BreKalKenOza17}, a countable group $G$ has no nontrivial normal amenable subgroups
if and only if its reduced C$^*$-algebra $C_\lambda^* (G)$ has a unique tracial state (the canonical one).
In Section~\ref{S-nonamenable 2}
we will construct large classes of topologically free minimal actions $\Gamma\curvearrowright X$
of countable amenable groups on the Cantor set whose alternating groups $\fA (\Gamma ,X)$
are simple and nonamenable. It follows that for such actions the reduced C$^*$-algebra
$C_\lambda^* (\fA (\Gamma ,X))$ has a unique tracial state, and hence has the C$^*$-algebraic
version of property Gamma \cite{GonJiaSu00,CasEviTikWhiWin18} as a consequence Theorem~\ref{T-Gamma} and
the fact that the von Neumann algebra of a discrete group is the weak operator closure of the
reduced group C$^*$-algebra under the GNS representation of the canonical tracial state.
\end{remark}

\begin{question}
Can one strengthen the conclusion of Theorem~\ref{T-Gamma} by replacing property Gamma
with the property that the group von Neumann algebra is McDuff?
\end{question}

\section{Inner amenability}\label{S-inner amenability}

A {\it mean} on a set $X$ is a finitely additive probability measure defined on the collection
of all subsets of $X$. An action of a discrete group $G$ on a set $X$ is called {\it amenable}
if there exists a $G$-invariant mean on $X$. We say that a discrete group $G$ is {\it inner amenable}
if there exists a conjugation invariant mean $m$ on $G$ which is atomless,
i.e., $m(F)=0$ for every finite subset $F$ of $G$. This definition of inner amenability is slightly stronger than Effros's
original one in \cite{Eff75} and coincides with it if the group is ICC (the main situation of interest in \cite{Eff75}),
but has become standard in the setting of general groups.

A mean on a set $X$ give rise via integration to a unital positive linear functional $\ell^\infty (X) \to \C$
(also called a {\it mean}),
and conversely each such functional produces a mean on $X$ by evaluation on indicator functions.
Thus inner amenability for a discrete group $G$ can be expressed as the existence of a
mean $\sigma : \ell^\infty (G) \to \C$ which vanishes on $c_0 (G)$ and
is invariant under the action of $G$ that composes a function
with the action $(s,t)\mapsto s^{-1} ts$ of $G$ on itself by conjugation.
It is moreover the case that $G$ is inner amenable if and only if for every finite set $F\subseteq G$ and $\eps > 0$
there exist finite sets $W\subseteq G$ of arbitraily large cardinality such that $|sWs^{-1} \triangle W| < \eps |W|$ for all $s\in F$.
The backward implication is established by viewing
the normalized indicator functions of the sets $W$ as means on $\ell^\infty (G)$
via the canonical inclusion
$\ell^1 (G) \subseteq \ell^\infty (G)^*$ and then taking a weak$^*$ cluster point over the net of pairs $(F,\eps )$
to produce a mean on $\ell^\infty (G)$ that witnesses inner amenability.

By \cite{Eff75}, if $G$ is an ICC discrete group with property Gamma then it is inner amenable.
We thereby derive the following corollary from Theorem~\ref{T-Gamma}.

\begin{corollary}\label{C-inner}
Let $\Gamma\curvearrowright X$ be a topologically free action
of an infinite amenable group on the Cantor set. Then every subgroup $G$ of $[[\Gamma\curvearrowright X ]]$ containing $\fA (\Gamma ,X)$ is inner amenable.
\end{corollary}

This corollary can also be proved more directly, and we will now describe two methods for doing this
that may be useful for establishing inner amenability in other contexts.
The first of these will actually give us the conclusion even without the assumption that the action is topologically free, but just under the (necessary) assumption that the group $G$ is infinite, which always holds when $\fA (\Gamma ,X )$ is nontrivial.

The following two folklore facts can be verified as exercises.

\begin{lemma}\label{lem:Folklore1}
Suppose that $G\curvearrowright X$ is an amenable action of a group $G$ on a set $X$.
Let $G\curvearrowright Y$ be another action of $G$ and assume that $G_x\curvearrowright Y$ is amenable for every $x\in X$.
Then $G\curvearrowright Y$ is amenable.
\end{lemma}

\begin{lemma}\label{lem:Folklore2}
Let $G$ be a group and let $H$ be a subgroup of $G$. Suppose that $H$ is inner amenable
and the action $G\curvearrowright G/H$ is amenable. Then $G$ is inner amenable.
\end{lemma}

Let $\Gamma \curvearrowright X$ be an action of a countable amenable group on the Cantor set and let $G$ be an infinite subgroup of $[[\Gamma \curvearrowright X ]]$ containing $\fA (\Gamma , X)$.  As usual, for a finite set $F\subseteq X$ we write $G_{(F)}$ for the set of all elements in $G$
which fix a neighborhood of every $x\in F$, and $G_{(x)}$ when $F$ is a singleton $\{ x \}$.

\begin{lemma}\label{lem:GFamen}
Let $F$ be a finite subset of $X$. Then the action of $G$ on $G/G_{(F)}$ is amenable.
\end{lemma}

\begin{proof}
Let $G_F=\bigcap _{x\in F}G_x$ be the pointwise stabilizer of $F$. Then the group $G_F /G_{(F)}$ is amenable, since it is isomorphic to a subgroup of $\prod _{x\in F}\Gamma _x/\Gamma _{(x)}$, which is amenable since $\Gamma$ is amenable. It follows that the action $G_F\curvearrowright G/G_{(F)}$ is amenable, and hence the action $gG_Fg^{-1} \curvearrowright G/G_{(F)}$ is amenable for each $g\in G$. Therefore, by Lemma \ref{lem:Folklore1}, to show that the action of $G$ on $G/G_{(F)}$ is amenable it is enough to show that the action $G\curvearrowright G/G_F$ is amenable. Let $F_0\subseteq F$ be the subset of points of $F$ whose $\Gamma$-orbit is infinite. Then $G_{F}$ has finite index in $G_{F_0}$, so it is enough to show that the action $G\curvearrowright G/G_{F_0}$ is amenable, since any $G$-invariant mean $m_0$ on $G/G_{F_0}$ lifts to a $G$-invariant mean $m$ on $G/G_F$ defined on $f\in \ell ^{\infty}(G/G_F)$ by
\[
\int _{G/G_F} f(gG_F) \, dm (gG_F) = \frac{1}{[G_{F_0} : G_{F}]}\int _{G/G_{F_0}} \sum _{gG_{F}\subseteq hG_{F_0}} f(gG_F) \, dm_0(hG_{F_0}) .
\]
Enumerate the elements of $F_0$ as $z_0,z_1,\dots ,z_{n-1}$. Consider the diagonal action of $G$ on $X^n$ given by $g(x_0,\dots , x_{n-1}) := (gx_0,\dots , gx_{n-1})$. This action is by homeomorphisms which are contained in the topological full group of the action of $\Gamma ^n$ on $X^n$. Let $Y \subseteq X^n$ denote the set of $n$-tuples of pairwise distinct points. Then the $G$-orbit of $(z_0,\dots , z_{n-1})$ coincides with the intersection of $Y$ with the $\Gamma ^n$-orbit, $\prod _{i<n}\Gamma z_i$, of $(z_0,\dots , z_{n-1})$. The action of $G$ on $G/G_{F_0}$ is then conjugate to the action of $G$ on the intersection of $Y$ with $\prod _{i<n}\Gamma z_i$. For each $i<n$ let $m_i$ be a mean on $\Gamma z_i$ which is invariant under the action of $\Gamma$. Then, arguing by induction on $n$, we see that the mean $m$ on $\prod _{i<n}\Gamma z_i$, defined on $f\in \ell ^{\infty}(\prod _{i<n}\Gamma z_i )$ by
\[
\int f(x_0,\dots , x_{n-1}) \, dm = \int _{\Gamma z_0}\cdots \int _{\Gamma z_{n-1}} f(x_0,\dots , x_{n-1}) \, dm _{n-1}(x_{n-1})\cdots dm _0 (x_0),
\]
is $\Gamma ^n$-invariant and satisfies $m (Y\cap \prod _{i<n}\Gamma z_i ) = 1$ since each orbit $\Gamma z_i$ is infinite. Since the action of $G$ on $Y\cap \prod _{i<n}\Gamma z_i$ is by piecewise translations of elements of $\Gamma ^n$, the mean $m$ witnesses that this action of $G$ is amenable. Therefore the action of $G$ on $G/G_{F_0}$ is amenable.
\end{proof}

Observe that a group which admits a nontrivial asymptotically central sequence $(c_n)_{n\in \N}$
is inner amenable, since any atomless mean on the set $\{ c_n : n\in\N \}$ will be conjugation invariant. We can thus assert the following in view of (iii) of Lemma~\ref{lem:IAfullgroup}.

\begin{lemma}\label{L-acs}
Let $X_0$ denote the set of points $x\in X$ with the property that for every open neighborhood $U$ of $x$ and every $n\in \N$ there exists some $y\in X$ such that $|\Gamma y \cap U|\geq n$. Let $F$ be a finite subset of $X$ containing a point from $X_0$. Then $G_{(F)}$ is inner amenable.
\end{lemma}

We can now obtain the inner amenability of $G$ as follows. If there is a uniform finite bound on the size of $\Gamma$-orbits in $X$ then $G$ is locally finite by (i) of Lemma \ref{lem:IAfullgroup}, and hence $G$ is inner amenable in this case, since it is infinite and amenable. So assume that there is no such uniform bound. Then the set $X_0$ is nonempty by Lemma \ref{lem:IAfullgroup}. Therefore, fixing $x\in X_0$, the group $G_{(x)}$ is inner amenable by Lemma~\ref{L-acs}, so by taking $F$ in Lemma \ref{lem:GFamen} to be $\{ x \}$, and applying Lemma~\ref{lem:Folklore2} with $H=G_{(x)}$, we conclude that $G$ is inner amenable.

\medskip

Another way to verify the inner amenability in Corollary~\ref{C-inner} is as follows.
Let $\Omega$ be a finite symmetric subset of $[[ \alpha ]]$ and $\eps > 0$.
It is enough to construct a finite set $W\subseteq \fA (\alpha )$ of cardinality at least $2$ such that
$|h W h^{-1} \triangle W | / |W| < \eps$ for all $h\in\Omega$.

By the definition of the topological full group,
we can find a clopen partition $\{ P_1 , \dots , P_m \}$ of $X$
such that for each $h\in\Omega$ there exist $s_{h,1} , \dots , s_{h,m} \in \Gamma$ for
which $hx = s_{h,i} x$ for every $i=1,\dots , m$ and $x\in P_i$.
Let $\eps' > 0$, to be determined.
By the amenability of $\Gamma$ there exists a nonempty finite set
$F\subseteq \Gamma$ such that the set
$F' := F\cap \bigcap_{h\in\Omega} \bigcap_{i=1}^m (s_{h,i} F \cap s_{h,i}^{-1} F)$
satisfies $|F' |\geq (1-\eps' )|F|$. Since $G$ is infinite, we may also assume that $|F|$
is large enough for a purpose below.
As the action is topologically free there exist an $x\in X$ such that the map $t\mapsto tx$
from $F$ to $X$ is injective. By continuity we can then find a clopen neighborhood
$B$ of $x$ such that the sets $tB$ for $t\in F$ are pairwise disjoint and
for each $t\in F$ there is an $1\leq i\leq m$ such that $tB \subseteq P_i$.

Write $\fA_F$ for the alternating group on the set $F$.
To each $\sigma\in\fA_F$ we associate
an element $g_\sigma \in \fA (\alpha )$ by setting
$g_\sigma sx = \sigma (s)x$ for all $s\in F$ and $x\in B$
and $g_\sigma x = x$ for all $x\in X\setminus FB$.
Write $W$ for the set of $g_\sigma$ such that $\sigma$ is a $3$-cycle in $\fA_F$,
and $W'$ for the set of $g_\sigma$ such that $\sigma$ is a $3$-cycle in $\fA_F$
with support contained in $F'$. Then
\begin{align*}
\frac{|W'|}{|W|}
= \frac{2\binom{|F'|}{3}}{2\binom{|F|}{3}}
= \frac{|F'|(|F'|-1)(|F'|-2)}{|F|(|F|-1)(|F|-2)} ,
\end{align*}
and so by taking $\eps'$ small enough and requiring $|F|$ to be large enough
we may ensure that $|W'|/|W| > 1-\eps /2$.

Now let $h\in\Omega$.
Let $\sigma = (t_1 \ t_2 \ t_3 )$ be a $3$-cycle in $\fA_F$
whose support $\{ t_1 , t_2 , t_3 \}$ is contained in $F'$.
Let $1\leq i_1 , i_2 , i_3 \leq m$ be such that $t_k B \subseteq P_{i_k}$ for $k=1,2,3$.
Then
\begin{align*}
hg_\sigma h^{-1} (s_{h,i_1} t_1 B)
= hg_\sigma h^{-1} (h t_1 B)
= ht_2 t_1^{-1} (t_1 B)
= ht_2 B
= s_{h,i_2} t_2 B
\end{align*}
and similarly $hg_\sigma h^{-1} (s_{h,i_2} t_2 B) = s_{h,i_3} t_3 B$
and $hg_\sigma h^{-1} (s_{h,i_3} t_3 B) = s_{h,i_1} t_1 B$.
The elements $s_{h,i_1} t_1$, $s_{h,i_2} t_2$, and $s_{h,i_3} t_3$ belong to $F$
and hence are distinct given that the sets $tB$ for $t\in F$ are pairwise disjoint and
$h$ is bijective.
Moreover, $hg_\sigma h^{-1}$ acts as the identity off of
$s_{h,i_1} t_1 B \sqcup s_{h,i_2} t_2 B \sqcup s_{h,i_3} t_3 B$. Thus $hg_\sigma h^{-1} = g_\omega$
where $\omega$ is the $3$-cycle in $\fA_F$ which cyclically permutes the elements
$s_{h,i_1} t_1$, $s_{h,i_2} t_2$, and $s_{h,i_3} t_3$ in that order.
We have thus shown that $hW' h^{-1} \subseteq W$, so that
\begin{align*}
|hWh^{-1} \triangle W|
\leq |(hWh^{-1} \cup W)\setminus hW' h^{-1} |
&\leq |hWh^{-1} | + |W| - 2|hW' h^{-1}| \\
&= 2(|W| - |W'|) \\
&< \eps |W| ,
\end{align*}
as desired.

\section{Entropy and continuous orbit equivalence}\label{S-entropy}

In order to show the pairwise nonisomorphism of the alternating groups in Theorem~\ref{T-nonamenable}
of the next section, we will need Theorem~\ref{T-coe}, which says that topological entropy
is an invariant of continuous orbit equivalence for topologically free actions of finitely generated amenable groups
which are not virtually cyclic.
When the actions are free, this result is a consequence of Austin's work on
measure entropy and bounded orbit equivalence \cite{Aus16} in conjunction with the variational principle,
as freeness guarantees that the ergodic p.m.p.\ actions appearing in the variational principle are free,
so that \cite{Aus16} applies. It also follows from \cite{Aus16} when the groups are
torsion-free and the actions are minimal and have nonzero topological entropy,
for in this case the ergodic p.m.p.\ actions with nonzero entropy appearing in the
variational principle are again guaranteed to be free, this time by \cite{Wei03,Mey16}
(in both of these cases one doesn't need to assume that the groups are not virtually cyclic,
in accord with \cite{Aus16}).
This second situation is in fact all that we need for our application in Theorem~\ref{T-ent nonamenable}, but we
include a proof here to illustrate how the geometric ideas in Austin's approach can be directly applied
to the topological framework.

We begin by recalling the definition of topological entropy (see Section~9.9 of \cite{KerLi16} for more details).
Let $\Gamma$ be a countable amenable group and $\Gamma\stackrel{\alpha}{\curvearrowright} X$ an action
on a compact metrizable space. Let $d$ be a compatible metric on $X$. For a finite set $F\subseteq \Gamma$ and $\eps > 0$,
a set $A\subseteq X$ is said to be {\it $(d,\alpha ,F,\eps )$-separated} if for all distinct $x,y\in A$ there
exists an $s\in F$ such that $d(sx,sy)>\eps$. Write $\sep_d (\alpha ,F,\eps )$ for the maximum cardinality
of a $(d,\alpha ,F,\eps )$-separated subset of $X$.

Let $(F_n )$ be a F{\o}lner sequence for $\Gamma$.
The {\it topological entropy} of $\alpha$ is defined by
\begin{align*}
\htopol (\alpha ) = \sup_{\eps > 0} \limsup_{n\to\infty} \frac{1}{|F_n|} \log\sep_d (\alpha ,F_n ,\eps ) .
\end{align*}
This does not depend on the choice of F{\o}lner sequence, nor on the choice of compatible metric.

As in Section~\ref{S-Gamma}, for finite sets $E,K\subseteq \Gamma$ and a $\delta > 0$ we say that $E$ is
{\it $(K,\delta )$-invariant} if $| E \cap \bigcap_{s\in K} s^{-1} E| \geq (1-\delta )|E|$, which is equivalent
to $| \{ s\in E : Ks \subseteq E \} | \geq (1-\delta )|E|$.

\begin{lemma}\label{L-invariant}
Let $\Gamma\stackrel{\alpha}{\curvearrowright} X$ and $\Lambda\stackrel{\beta}{\curvearrowright} Y$
be continuous actions of amenable groups on compact metrizable spaces
and $\Phi : X\to Y$ a continuous orbit equivalence between $\alpha$ and $\beta$.
Let $\kappa$ and $\lambda$ be as in Definition~\ref{D-coe} with respect to $\Phi$.
Let $X_0$ be a $\Gamma$-invariant subset of $X$ such that $\Gamma$ acts freely on $X_0$ and 
$\Lambda$ acts freely on $\Phi (X_0 )$.
Let $L$ be a finite subset of $\Lambda$ and $\delta > 0$.
Let $F$ be a nonempty $(\lambda (L,Y),\delta )$-invariant finite subset of $\Gamma$.
Then for every $x\in X_0$ the subset $\kappa (F,x)$ of $\Lambda$ is $(L,\delta )$-invariant.
\end{lemma}

\begin{proof}
We may assume, by conjugating $\beta$ by $\Phi$, that $Y=X$ and $\Phi = \id_X$.
Let $x\in X_0$.
By the properties of $\kappa$ and $\lambda$, for every $t\in \Gamma$ and $h\in L$ we have
\begin{align*}
\beta_{h\kappa (t,x)} x = \beta_h \beta_{\kappa (t,x)} x = \beta_h \alpha_t x
= \alpha_{\lambda (h,\alpha_t x)} \alpha_t x = \alpha_{\lambda (h,\alpha_t x)t}x \in \alpha_{\lambda (L,X)t} x
\end{align*}
and hence $\beta_{L\kappa (t,x)} x \subseteq \alpha_{\lambda (L,X)t} x$.
By our freeness hypothesis the map
$t\mapsto \kappa (t,x)$ from $\Gamma$ to $\Lambda$ is bijective, and so we obtain
\begin{align*}
| \{ h\in \kappa (F,x) : Lh\subseteq \kappa (F,x) \} |
&\geq | \{ s\in F : \lambda (L,X)s\subseteq F \} | \\
&\geq (1-\delta )|F|
= (1-\delta )|\kappa (F,x)| ,
\end{align*}
that is, $\kappa (F,x)$ is $(L,\delta )$-invariant.
\end{proof}

Suppose now that $\Gamma$ finitely generated and $S$ is a symmetric generating set for $\Gamma$.
We equip $\Gamma$ with the right-invariant word metric associated to $S$,
whose value for a pair $(s,t)\in \Gamma\times \Gamma$
is the least $r\in\N$ such that $st^{-1} \in S^r$.

\begin{definition}\label{D-connected}
Let $V$ be a subset of $\Gamma$. We say $V$ is {\it connected}
if for all $v,w\in V$ there exist $s_1 ,\dots , s_n \in S$ such that
$s_n s_{n-1}\cdots s_1 v = w$ and $s_k s_{k-1} \cdots s_1 v \in V$ for every $k=1,\dots , n-1$.
For $r\in\N$, we say that $V$ is {\it $r$-separated} if the distance between any two distinct
elements of $V$ is greater than $r$, and that
a subset $C$ of $V$ is {\it $r$-spanning} for $V$ if for every
$v\in V$ there is a $w\in C$ such that
the distance between $v$ and $w$ is at most $r$.
\end{definition}

The following is Lemma~8.3 of \cite{Aus16}.

\begin{lemma}\label{L-connected}
For every finite set $F\subseteq \Gamma$ and $\delta > 0$ there exists a 
nonempty connected $(F,\delta )$-invariant finite subset of $\Gamma$.
\end{lemma}

The following is a slight variation of Corollary~8.6 of \cite{Aus16} and follows by the same argument,
which we reproduce below mutatis mutandis.

\begin{lemma}\label{L-tree}
Suppose that $\Gamma$ is not virtually cyclic.
Then there is a constant $b>0$ such that for every $r\in\N$
and every connected finite set $F\subseteq \Gamma$ with $|S^r F|\leq 2|F|$
there exists a $2r$-spanning tree in $S^r F$ with at most $b|F|/r$ vertices.
\end{lemma}

\begin{proof}
Since $\Gamma$ is not virtually cyclic there exists
a $c>0$ such that $|S^r | \geq cr^2$ for all $r\in\N$ (Corollary~3.5 of \cite{Man12}).
Suppose now that we are given an $r\in\N$ and a nonempty 
connected finite set $F\subseteq \Gamma$ with $|S^r F|\leq 2|F|$.
Take a maximal $r$-separated set $V\subseteq F$. Then
\begin{align*}
|V||S^r| = \Big| \bigsqcup_{g\in V} S^r g\Big| \leq |S^r F| \leq 2|F|
\end{align*}
and so
\begin{align*}
|V| \leq \frac{2|F|}{|S^r|} \leq \frac{2}{cr^2} |F| .
\end{align*}
By maximality, $V$ is $2r$-spanning for $F$. Consider now the graph on $V$ whose edges are those
pairs of points which can be joined, within $S^r F$, by an $S$-path of length at most $4r+1$.
This graph is connected because $F$ is connected
(given any two vertices in $V$, join them by an $S$-path in $F$ and then lengthen this path
within $S^r F$ by inserting at each of its non-endpoint vertices $w$
a loop consisting of a path of length at most $r$ joining $w$ to a vertex in $V$,
followed by the reversal of this path).
We can thus find a spanning tree $(V,E)$ within the graph.
For each pair $(s,t)$ in $E$ choose an $S$-path in $S^r V$ 
of length at most $4r+1$ from $s$ to $t$ and write $E'$ for
the collection of all pairs that appear as an edge in one of these paths. The graph consisting of all of the vertices
in $E'$ and with $E'$ as its edge set is connected and hence contains a spanning tree, and 
if $W$ denotes the vertex set of this spanning tree then 
\begin{align*}
|W| \leq |E'| + 1\leq (4r+1)|E| + 1 \leq 5r|V| \leq 5r\cdot \frac{2}{cr^2} |F| = \frac{10}{cr} |F| .
\end{align*}
We may thus take $b = 10/c$.
\end{proof}

\begin{theorem}\label{T-coe}
Let $\Gamma$ and $\Lambda$ be finitely generated amenable groups which are not 
virtually cyclic and let
$\Gamma\stackrel{\alpha}{\curvearrowright} X$ and $\Lambda\stackrel{\beta}{\curvearrowright} Y$
be topologically free actions on compact metrizable spaces which are continuously orbit equivalent.
Then $\htopol (\alpha ) = \htopol (\beta )$.
\end{theorem}

\begin{proof}
We show that $\htopol (\alpha ) \leq \htopol (\beta )$, which by symmetry is enough to prove the theorem.
We may assume that $\htopol (\alpha ) > 0$. We will also assume that
$\htopol (\alpha ) < \infty$, the argument for $\htopol (\alpha ) = \infty$ being the same
subject to obvious modifications.

We may assume that $Y=X$ and that $\id_X$ is a continuous orbit equivalence
by conjugating $\beta$ by a continuous orbit equivalence between $\alpha$ and $\beta$.
Let $\kappa : \Gamma\times X\to \Lambda$ and $\lambda : \Lambda\times X\to \Gamma$ be the associated cocycles.
Let $X_0$ be the dense $G_\delta$ subset of $X$ consisting of those points which
have trivial stabilizer for $\alpha$ and $\beta$. This set is invariant for both $\alpha$ and $\beta$,
as these actions share the same orbits.
Fix a compatible metric $d$ on $X$.

Set $h = \htopol (\alpha )$ for brevity. Let $\eta > 0$.
Fix a finite symmetric generating set $S$ for $\Gamma$ containing $1_\Gamma$ and endow 
$\Gamma$ with the associated word-length metric. Take a $b>0$ as given by Lemma~\ref{L-tree}
with respect to this metric. Since $\kappa$ is continuous we can find a clopen partition $\cP$ of $X$
such that for every $s\in S$ the map $x\mapsto \kappa (s,x)$ is constant on each member of $\cP$.
Choose an $r\in\N$ large enough so that $(b/r)\log |\cP | \leq\eta$.

Let $\eps > 0$. 
Let $L$ be a finite symmetric generating set for $\Lambda$ and let $\delta > 0$.
Take a nonempty finite set $F\subseteq \Gamma$ which is sufficiently left invariant so that
the set $F' := S^r F$ is $(\lambda (L,X),\delta )$-invariant and has cardinality at most $2|F|$, 
and also chosen so that $|F'|^{-1} \log \sep_d (\alpha ,F',\eps ) > h - \eta$, which is possible
since we can force $F'$ to be as left invariant as we wish.
By Lemma~\ref{L-connected} we may assume that $F$ is connected.

By uniform continuity there is an $\eps' > 0$ such that
if $x,y\in X$ satisfy $d(x,y) < \eps'$ then $d(sx,sy) < \eps$ for all $s\in S^{3r}$.
By our choice of $F$ there exists a $(d ,\alpha ,F',\eps )$-separated set $A\subseteq X$
of cardinality at least $e^{(h-\eta )|F'|}$, and by making small perturbations we may assume that $A\subseteq X_0$.

By our choice of $b$ and $F$ there exists a $2r$-spanning tree in $F'$ whose vertex set $V$ satisfies $|V|\leq b|F|/r$.
Root this tree at some arbitrarily chosen vertex $v_0$. 
By replacing $F$ with $Fv_0^{-1}$ and $A$ with its image under $\alpha_{v_0}$,
we may assume that $v_0 = 1_\Gamma$.
Consider a directed path $(v_0 , v_1 , \dots , v_n )$
in the tree which starts at $v_0 = 1_\Gamma$. 
Set $s_k = v_k v_{k-1}^{-1} \in S$ for $k=1,\dots ,n$. Making repeated use of
the cocycle identity we have, for every $x\in X$,
\begin{align*}
\kappa (v_n ,x)
= \kappa (s_n \cdots s_1 ,x)
= \prod_{k=0}^{n-1} \kappa (s_{n-k} , s_{n-k-1} \cdots s_1 x) .
\end{align*}
This shows that the collection of all maps from $V$ to $\Lambda$ of the form $v\mapsto\kappa (v,x)$ for some $x\in X$
has cardinality at most $|\cP |^{|V|}$, which is bounded above by $|\cP |^{b|F|/r} = e^{(b/r)(\log |\cP |)|F|}$,
which in turn is bounded above by $e^{\eta |F|}$ by our choice of $r$.
It follows by the pigeonhole principle that we can find an $A_0 \subseteq A$
with $|A_0| \geq e^{-\eta |F|} |A| \geq e^{(h-2\eta )|F'|}$ such that
the map $v\mapsto \kappa (v,x)$ from $V$ to $\Lambda$ is the same for every $x\in A_0$. Write $W$
for the common image of these maps.

Choose an $x_0 \in A_0$ and set $E = \kappa (F',x_0 )\subseteq\Lambda$. 
Since $A_0 \subseteq X_0$ and $\alpha$ and $\beta$ act freely on $X_0$, 
the set $E$ is $(L,\delta )$-invariant by Lemma~\ref{L-invariant}
and has the same cardinality as $F'$. Since $A_0$ is $(d ,\alpha ,F',\eps )$-separated
and $V$ is $2r$-dense in $F$ and hence $3r$-dense in $F'$, 
the set $A_0$ is $(d ,\alpha ,V,\eps' )$-separated, and hence also
$(d ,\beta ,W,\eps' )$-separated.
Consequently
\begin{align*}
\frac{1}{|E|} \log\sep_d (\beta ,E,\eps' )
\geq \frac{1}{|E|} \log\sep_d (\beta ,W,\eps' )
\geq \frac{1}{|F'|} \log |A_0 |
\geq h-2\eta .
\end{align*}
Since $L$ is a symmetric generating set we can force $E$ to be as left invariant as we wish
by taking $\delta$ sufficiently small, and since $\eps'$ does not depend on $\delta$
we thus obtain $\htopol (\beta ) \geq h-2\eta$.
Letting $\eta\to 0$ we conclude that $\htopol (\beta ) \geq h$, as desired.
\end{proof}

\section{C$^*$-simple alternating groups}\label{S-nonamenable 2}

We begin by formulating a property of a countable group $\Gamma$, called property ID,
that will permit us to carry out the construction
in the proof of Theorem~\ref{T-nonamenable}. Among the $\Gamma$ which satisfy this property
are those that are residually finite and contain a nontorsion element
with infinite conjugacy class, as we show in Theorem~\ref{T-res finite}.

By a {\it tiling} of $\Gamma$ we mean a finite collection $\{ (S_i ,C_i ) \}_{i\in I}$ of pairs where
the $S_i$ are nonempty finite subsets of $\Gamma$
({\it shapes})
and the $C_i$ are subsets of $\Gamma$ ({\it center sets})
such that $\Gamma$ partitions as $\bigsqcup_{i\in I} \bigsqcup_{c\in C_i} S_i c$.
The sets $S_i c$ in this disjoint union are called {\it tiles}.
The tiling is a {\it monotiling} if the index set $I$ is a singleton.

By a {\it tightly nested} sequence $(\cT_n )$ of tilings we mean that, for every $n>1$ and $(S,C)\in\cT_n$,
the shape $S$ can be expressed as a partition $\bigsqcup_{j\in J} S_j g_j$,
where the $S_j$ are shapes of $\cT_{n-1}$ and the $g_j$ are elements of $\Gamma$,
such that for every $c\in C$ and $j\in J$ the element $g_j c$ is a tiling center for the shape $S_j$
(in which case $\bigsqcup_{j\in J} S_j g_j c$ is a partition of $Sc$ into tiles of $\cT_{n-1}$).
Observe that if the sequence $(\cT_n )$ is tightly nested then all of its subsequences are as well.

We call $(\cT_n )$ a {\it sequence of F{\o}lner tilings} if for every finite set $F\subseteq \Gamma$
and $\delta > 0$ there is an $N\in\N$ such that for every $n\geq N$ each shape $S$ of the
tiling $\cT_n$ is $(F,\delta )$-invariant, i.e.,
$|S\cap\bigcap_{t\in F} St^{-1} | \geq (1-\delta )|S|$.

As in Section~3 of \cite{DowHucZha16},
to each tiling $\{ (S_i , C_i ) \}_{i\in I}$ of $\Gamma$ we associate an action $\alpha$
of $\Gamma$ as follows. Write $Y = \{ S_i : i\in I \} \cup \{ 0 \}$
and define $x\in Y^\Gamma$ by $x_t = S_i$ if $t\in C_i$
and $x_t = 0$ otherwise.
The action $\alpha$ is then the restriction of the right shift action
$\Gamma\curvearrowright Y^\Gamma$
to the closed $\Gamma$-invariant set $\overline{\Gamma x}$.
The {\it entropy} of the tiling is defined to be the topological entropy of this subshift action.

A set $C\subseteq\Gamma$ is said to be {\it syndetic} if there is a finite set $F\subseteq\Gamma$
for which $FC = \Gamma$.

\begin{definition}
We say that $\Gamma$ has {\it property ID} if there exist an infinite cyclic subgroup $H\subseteq\Gamma$
of infinite index and a tightly nested sequence $(\cT_n )$ of F{\o}lner tilings with syndetic center sets
and entropies converging to zero such that
\begin{enumerate}
\item for every finite set $F\subseteq \Gamma$ there is an $n\in\N$ such that
$F$ is contained in some tile of $\cT_n$,

\item for every $\delta > 0$ there
is an $n_0 \in\N$ such that for all $n\geq n_0$ one has $|H\cap T| \leq \delta |T|$
for every tile $T$ of $\cT_n$, and

\item for every $n\in\N$ there are a shape $S$ of $\cT_n$ and a subgroup $H_0$ of $H$ such that
the sets $Sc$ for $c\in H_0$ are tiles of $\cT_n$ whose union contains $H$.
\end{enumerate}
\end{definition}

\begin{example}
$\Z^d$ for $d\geq 2$ is the prototype of a group with property ID.
\end{example}

\begin{example}\label{E-product}
More generally, if $\Gamma$ is a countable amenable group of the form $\Gamma_0 \times \Z$ where $\Gamma_0$ is infinite,
then $\Gamma$ has property ID. To verify this we can use tilings which are products (in the obvious sense)
of a suitable tiling of $\Gamma_0$, as given by Theorem~3.5 of \cite{Dou17},
and a monotiling of $\Z$ by translates of an interval of the form $\{ -n,-n+1,\dots ,n \}$.
\end{example}

In another direction of generalization we have Theorem~\ref{T-res finite} below.
To establish this we need the following lemma.

\begin{lemma}\label{L-zero}
Let $N_1 \supseteq N_2 \supseteq\dots$ be a decreasing sequence of finite-index normal subgroups of $\Gamma$
with intersection $\{ 1_\Gamma \}$.
Let $a$ be an element of $\Gamma$ with infinite conjugacy class, and set $H = \langle a \rangle$.
Then $|\Gamma/N_k |/|HN_k /N_k | \to \infty$ as $k\to\infty$.
\end{lemma}

\begin{proof}
We may naturally identify $\Gamma$ with a subgroup of the profinite group $P = \varprojlim \Gamma /N_k$,
in which case the limit $\lim _{k\to\infty}|\Gamma /N_k |/|HN_k / N_k |$ coincides with the index of the closure, $\overline{H}$,
of $H$ in $P$. The group $\overline{H}$ is abelian since $H$ is, and so if $\overline{H}$ were of finite index in $P$
then $a$ would commute with the finite-index subgroup $\Gamma \cap \overline{H}$ of $\Gamma$, which is impossible since
the conjugacy class of $a$ in $\Gamma$ is infinite.
\end{proof}

\begin{theorem}\label{T-res finite}
Let $\Gamma$ be a residually finite countable amenable group that contains a nontorsion element
with infinite conjugacy class.
Then $\Gamma$ has property ID.
\end{theorem}

\begin{proof}
Let $a$ be a nonorsion element of $\Gamma$ with infinite conjugacy class,
and write $H$ for the subgroup of $\Gamma$ it generates.
Fix an increasing sequence $\{ 1_\Gamma \} = E_1 \subseteq E_2 \subseteq\dots$ of finite subsets of $\Gamma$
whose union is equal to $\Gamma$, and a decreasing sequence $1 = \eps_1 > \eps_2 > \dots$ of strictly
positive real numbers tending to zero.
We will recursively construct a tightly nested sequence of monotilings $\cT_k = \{ (S_k , N_k ) \}$
such that $N_1 \supseteq N_2 \supseteq\dots$ is a decreasing sequence of finite-index normal subgroups of $\Gamma$
with intersection $\{ 1_\Gamma \}$ and for every $k$ the shape $S_k$ is $(E_k ,\eps_k )$-invariant
and contains $E_k$, the pair $\{ (S_k \cap H, N_k \cap H ) \}$ is a monotiling of $H$,
and $|H \cap S_k c | \leq \eps_k |S_k |$ for every $c\in N_k$.
Since $N_k \cap H$ is a subgroup of $H$ for each $k$,
the unique center set of a monotiling is trivially syndetic, and the entropies
of a sequence of F{\o}lner monotilings must converge to zero (as can be seen by
a density argument and Stirling's formula---see Lemma~3.9 of \cite{DowHucZha16}),
this will establish property ID.

For the base case $k=1$ set $\cT_1 = \{ (\{ 1_\Gamma \} , \Gamma ) \}$.
Now let $k>1$ and assume that we have constructed $\cT_{k-1} = \{ (S_{k-1} , N_{k-1} ) \}$
with the desired properties.
As is readily seen, there exists a $\delta > 0$ with $\delta\leq\eps_k /2$ such that every
$(E_k ,\delta )$-invariant
nonempty finite set $F\subseteq \Gamma$ has the property that every set $F' \subseteq \Gamma$
satisfying $|F'| = |F|$ and $|F' \cap F| \geq (1-\delta )|F|$ is $(E_k , \eps_k )$-invariant.
Since $\Gamma$ is residually finite, by the quasitiling argument in Section~2 of \cite{Wei01} we can find a
finite-index normal subgroup $N_k \subseteq \Gamma$ with $N_k \subseteq N_{k-1}$ and
$N_k \cap (E_k \setminus \{ 1_\Gamma \} ) = \emptyset$ and
a transversal $F$ for $N_k$ in $\Gamma$ which contains $E_k$ and is $(E_k ,\delta )$-invariant
and is also sufficiently left invariant so that
\begin{align}\label{E-boundary}
|\{ c\in \Gamma : F\cap S_{k-1} c \neq\emptyset\text{ and } (\Gamma\setminus F)\cap S_{k-1} c \neq\emptyset \}|
\leq \delta |F|/|S_{k-1}| .
\end{align}
We may moreover assume, in view of Lemma~\ref{L-zero},
that $|HN_k /N_k |\leq \delta |\Gamma/N_k |$, which implies that
\begin{align}\label{E-shape 1}
|HN_k \cap F| \leq \delta |F| .
\end{align}

Define $S_k$ to be the set $S_{k-1} (F\cap N_{k-1} )$,
which is equal to the disjoint union
$\bigsqcup_{c\in F\cap N_{k-1}} S_{k-1} c$. Since $S_{k-1}$ is a transversal for $N_{k-1}$ in $\Gamma$
and $F\cap N_{k-1}$ is a transversal for $N_k$ in $N_{k-1}$,
it follows that $S_k$ is a transversal for $N_k$ in $\Gamma$, and so $\cT_k := \{ (S_k ,N_k ) \}$
is a monotiling of $\Gamma$.

For every element $t$ of $F$ not belonging to $S_{k-1} (F\cap N_{k-1} )$, the fact that the sets
$S_{k-1} c$ for $c\in N_{k-1}$ tile $\Gamma$ implies the existence of a $c_0 \in N_{k-1} \setminus F$
such that $t\in S_{k-1} c_0$, in which case $S_{k-1} c_0$ intersects both
$F$ and $\Gamma\setminus F$ (since $1_\Gamma \in S_{k-1}$).
In view of (\ref{E-boundary}) this implies that the set of all such $t$
has cardinality less than $\delta |F|$, so that
\begin{align}\label{E-shape 2}
|S_k \cap F|
= |S_{k-1} (F\cap N_{k-1} )\cap F|
\geq (1-\delta )|F| .
\end{align}
Since $|S_k | = |\Gamma /N_k | = |F|$ it follows by our choice of $\delta$ that
$S_k$ is $(E_k ,\eps_k )$-invariant.

Finally, observe using (\ref{E-shape 1}) and (\ref{E-shape 2}) that for all $c\in N_k$ we have
\begin{align*}
|H\cap S_k c|
\leq |HN_k \cap S_k |
\leq |HN_k \cap F| + |S_k \setminus F|
\leq 2\delta |F|
\leq \eps_k |F|
= \eps_k |S_k | .
\end{align*}
This completes the recursive construction.

It remains to observe that, since $S_1 \subseteq S_2 \subseteq\dots$ and the tiling center sets $N_k$
are normal subgroups, the sequence $(\cT_k )$ is tightly nested.
\end{proof}

We next aim to prove Theorem~\ref{T-nonamenable}.

\begin{lemma}\label{L-top free}
Let $\Gamma$ be a torsion-free countable amenable group and let $\Gamma\curvearrowright X$ be a
minimal action on a compact metrizable space with nonzero topological entropy.
Then the action is topologically free.
\end{lemma}

\begin{proof}
By the variational principle (Theorem~9.48 of \cite{KerLi16}) there is a
$\Gamma$-invariant Borel probability measure $\mu$ on $X$
such that the measure entropy of the action $\Gamma\curvearrowright (X,\mu )$
is nonzero. By \cite{Wei03,Mey16} and torsion-freeness,
there is a $\mu$-nonnull, and in particular nonempty,
set of points in $X$ whose stabilizer is trivial.
By minimality this implies that the action $\Gamma\curvearrowright X$ is topologically free.
\end{proof}

Given a group $G$ of transformations of a space $X$, the {\it rigid stablizer} $G_U$ of a
set $U\subseteq X$ is defined as the group of all elements in $G$ 
which fix every point in the complement of $U$.

\begin{theorem}\label{T-nonamenable}
Let $\Gamma$ be a torsion-free countable amenable group with property ID.
Let $\lambda\in (0,\infty )$. Then there is an expansive topologically free minimal action
$\alpha$ of $\Gamma$ on the Cantor set $X$ with topological entropy $\lambda$ such that
for every nonempty open set $U\subseteq X$ the free product $\Z_2 * \Z_2 * \Z_2$
embeds into the rigid stabilizer $\fA (\alpha )_U$.
\end{theorem}

\begin{proof}
Fix an infinite cyclic subgroup $H\subseteq \Gamma$ and a tightly nested sequence $(\cT_n )$ of F{\o}lner
tilings for $\Gamma$ as in the definition of Property ID.

Let $q$ be an integer greater than both $3$ and $e^{2\lambda}$.
For a set $L\subseteq \Gamma$ we will write $\pi _L$ for the coordinate projection map
$x\mapsto x|_L$ from $\{ 1,\dots , q\}^\Gamma$ to $\{ 1,\dots , q\}^L$.
By a {\it box} in $\{ 1,\dots ,q \}^\Gamma$ we mean a subset of $\{ 1,\dots ,q \}^\Gamma$
of the form $\prod_{s\in \Gamma} A_s$ where $A_s \subseteq \{ 1,\dots , q \}$ for each $s\in \Gamma$.
For every box $A = \prod_{t\in \Gamma} A_t \subseteq\{ 1,\dots ,q \}^\Gamma$ and $T\subseteq \Gamma$ write
\begin{align*}
D_T (A) = \{ t\in T : A_t = \{ 1,\dots ,q \} \} .
\end{align*}

Set $\theta = \lambda / \log q$, and note that $\theta < 1/2$ by our choice of $q$.
We will recursively construct an
increasing sequence $1 = n_0 < n_1 < n_2 <\dots$ of integers and a decreasing
sequence $A_0 = \{ 1,\dots ,q \}^\Gamma \supseteq
A_1 = \prod_{t\in \Gamma} A_{1,t} \supseteq A_2 = \prod_{t\in \Gamma} A_{2,t} \supseteq \dots$
of boxes in $\{ 1,\dots ,q \}^\Gamma$ such that, writing $F_k$ for the tile of $\cT_{n_k}$
containing $1_\Gamma$, we have the following for every $k \geq 1$:
\begin{enumerate}
\item $\theta + 1/2^{k+1} < |D_T (A_k)|/|T| < \theta + 1/2^k$
for every tile $T$ of $\cT_{n_k}$,

\item $|T|/2^{k+4} > 1$ for every tile $T$ of $\cT_{n_k}$,

\item $H \subseteq D_\Gamma (A_k)$,

\item for every $(S,C) \in \cT_{n_k}$ one has $A_{k,sc} = A_{k,sd}$ for all $s\in S$ and $c,d\in C$,

\item $\pi_{F_{k-1}} (A_k ) = \pi_{F_{k-1}} (A_{k-1} )$,
\end{enumerate}

As indicated we take $n_0 =1$ and $A_0 = \{ 1,\dots ,q \}^\Gamma$.
Suppose now that $k\geq 1$ and that we have defined $n_{k-1}$ and $A_{k-1}$.

Let $\{ S_j \}_{j\in J}$ be the collection of shapes of the tiling $\cT_{n_{k-1}}$ and for each $j$
choose a tiling center $c_j$ for $S_j$ and set $W_j = \prod_{s\in S_j} A_{k-1,sc_j}$
(this is independent of the choice of the $c_j$ since (iv) holds for $k-1$).

Take an integer $n_k > n_{k-1}$ which is large enough for purposes to be specified.
Pick a set $\cR$ of representatives among the tiles of $\cT_{n_k}$
for the relation of having the same shape, and include among these representatives
the tile $F_k$ which contains $1_\Gamma$.

Let $T\in\cR$. Then we have $T = \bigsqcup_{i\in I} T_i$
for some tiles $T_i$ from $\cT_{n_{k-1}}$. By taking $n_k$ large enough (independently
of the particular $T\in\cR$ at hand, which we can do because $\cR$ is finite) we may ensure
that $|T| /2^{k+4} > 1$ and that
$I$ can be partitioned as $I' \sqcup I'' \sqcup \bigsqcup_{j\in J} I_j$ where
\begin{enumerate}
\item[(a)] the tiles $T_i$ for $i\in I'$ cover a large enough proportion of $T$ so that
\begin{align*}
\sum_{i\in I'} |T_i| > \frac{\theta + 1/2^{k+1}}{\theta + 5/2^{k+3}}|T| ,
\end{align*}

\item[(b)] $I''$ is equal to the set of all $i\in I$ such that either $T_i = F_{k-1}$
or $T_i \cap H \neq \emptyset$ and the tiles $T_i$ for $i\in I''$ cover a small enough proportion of $T$
(as is possible by condition (ii) in the definition of property ID) so that
\begin{align*}
\sum_{i\in I''} |T_i| < \frac{1}{2^{k+3}}|T|
\end{align*}
and also so that we can arrange (c) below,

\item[(c)] for each $j\in J$ the tiles $T_i$ for $i\in I_j$ all have shape $S_j$
and there exists a surjection $\varphi_j : I_j \to W_j$ such that for every $i\in I_j$
we have $\varphi_j (i)_s \in A_{k-1,sc}$ for all $s\in S_j$
where $c$ is the tiling center for $T_i$
(this is possible view of (iv) for $k-1$, since the tiling center sets for $\cT_{n_{k-1}}$ are syndetic
and so $T$ can be chosen to be sufficiently left invariant so that its
intersection with each of these tiling center sets has sufficiently large cardinality both
to enable the existence of such $\varphi_j$ and to allow for (a), with the
tiles $T_i$ for $i\in I''$ being assumed at the same time to cover a small enough proportion of $T$).
\end{enumerate}
Using this partition of $I$ we define the sets $A_{k,t}$ for $t\in T$ in three stages:
\begin{enumerate}[label=\roman*]
\item[(d)] For $i\in I''$ and $t\in T_i$ define $A_{k,t} = A_{k-1,t}$.

\item[(e)] For each $j\in J$ and $i\in I_j$
define $A_{k,sc} = \varphi_j (i)_s$ for all $s\in S_j$ where $c$ is the tiling center for $T_i$.

\item[(f)] For each $i\in I'$ choose a set $T_i' \subseteq D_{T_i} (A_{k-1} )$ with
\begin{align*}
(\theta + 5/2^{k+3}) |T_i | < |T_i' | < (\theta + 6/2^{k+3})|T_i| ,
\end{align*}
which we can do since $|T_i|/2^{k+3} > 1$
and $|D_{T_i} (A_{k-1} )| > (\theta + 1/2^k )|T_i|$ by hypothesis
(including the case $k=1$ since $\theta < 1/2$).
We define $A_{k,t} = A_{k-1,t}$ for all $t\in T_i'$,
while for $t\in T_i \setminus T_i'$ we define $A_{k,t}$ to be any singleton contained in $A_{k-1,t}$.
\end{enumerate}
Note that by (b) and (f) we have
\begin{align}\label{E-density 1}
D_T (A_k )
< (\theta + 6/2^{k+3} ) \sum_{i\in I'} |T_i | + \sum_{i\in I''} |T_i |
< (\theta + 1/2^k ) |T|
\end{align}
while by (a) and (f) we have
\begin{align}\label{E-density 2}
D_T (A_k )
> ( \theta + 5/2^{k+3} ) \sum_{i\in I'} |T_i |
> ( \theta + 1/2^{k+1} ) |T| .
\end{align}

Having thereby defined $A_{k,t}$ for all $t$ in each tile in our collection $\cR$
of shape representatives, we can now extend the definition of $A_{k,t}$ to all $t\in \Gamma$
in the unique way that enables us to satisfy (iv). Since the sequence of tilings $(\cT_n )$ is tightly nested
and (iv) holds for $k-1$, we will have the inclusion
$A_{k,t} \subseteq A_{k-1,t}$ for every $t\in \Gamma$. Condition (i) for $k$ is verified by (\ref{E-density 1}) and (\ref{E-density 2}),
while conditions (ii) and (v) are built into the construction.
Since condition (iii) holds for $k-1$,
it is guaranteed to hold for $k$ by our construction.

Set $A = \bigcap_{k=1}^\infty A_k$, which is a decreasing intersection of nonempty closed sets
and hence is itself nonempty and closed. Consider the right shift action
$\Gamma\curvearrowright \{ 1,\dots ,q \}^\Gamma$ and set $X = \overline{\Gamma A}$, which is closed and
$\Gamma$-invariant. Let $\Gamma\stackrel{\alpha}{\curvearrowright} X$ be the subshift action.
It remains to check that $\alpha$ is topologically free and minimal, that
its topological entropy is equal to $\lambda$, and that $\Z_2 * \Z_2 * \Z_2$
embeds into $\fA (\alpha )_U$ for every nonempty open set $U\subseteq X$.
Note that minimality, topological freeness, and the infiniteness of $\Gamma$
together imply that the space
$X$ has no isolated points and hence, being a compact subset of a metrizable zero-dimensional space, is the Cantor set.

For minimality, let $x,y\in A$. Let $F$ be a nonempty finite subset of $\Gamma$.
Given that the sequence $(\cT_n )$ came from the definition of property ID, we can
find a $k$ such that $F$ is contained in a tile $T$ of $\cT_{n_k}$.
Let $(S,C)\in\cT_{n_{k+1}}$.
By the construction of the set $A_{k+1}$, for every $c \in C$ there is a $\tilde{c} \in \Gamma$
such that $T\tilde{c} \subseteq Sc$ and $y_{t\tilde{c}} = x_t$ for all $t\in T$,
in which case $\tilde{c} y|_F = x|_F$.
Also, since $C$ is syndetic by the definition of property ID,
there is a finite set $E\subseteq \Gamma$ such that $EC = \Gamma$, in which case the set
$\tilde{C} = \{ \tilde{c} : c\in C \}$
satisfies $ES^{-1} T \tilde{C} = \Gamma$ and hence is syndetic.
From these observations we deduce that $x\in\overline{\Gamma y}$ and,
by taking $y=x$ and applying
a standard characterization of minimality (Proposition~7.13 of \cite{KerLi16}), that
the action $\Gamma\curvearrowright \overline{\Gamma x}$ is minimal.
On the other hand, reversing the roles of $x$ and $y$ we get $y\in\overline{\Gamma x}$ whence
$\overline{\Gamma A} = \overline{\Gamma x}$ by the minimality of $\Gamma\curvearrowright \overline{\Gamma x}$,
so that $\alpha$ is minimal.

For the entropy calculation, an inductive application of (v) shows that
for every $k\in\N$ we have $\pi_{F_k} (A_j ) = \pi_{F_k} (A_k )$ for all $j\geq k$
and hence $\pi_{F_k} (A) = \pi_{F_k} (A_k )$, so that, by (i),
\begin{align*}
\frac{1}{|F_k|} \log |\pi_{F_k} (A)|
= \frac{|D_{F_k} (A_k )|}{|F_k|}\log q
&\geq (\theta + 1/2^{k+1} )\log q > \lambda .
\end{align*}
Since $(F_k )_k$ is a F{\o}lner sequence for $\Gamma$, it follows that
$\htopol (\Gamma,X) \geq \lambda$.

Let $\delta > 0$. Take a large enough $k\in\N$ so that
the tiling $\cT_{n_k}$ has entropy less than $\delta$, which we can do by the definition of property ID.
Writing $S$ for the union of the shapes of the tiling $\cT_{n_k}$, any sufficiently left invariant
nonempty finite set $F\subseteq \Gamma$ will be such that the set of restrictions to
$S^{-1} F$ of elements in the right shift action associated to $\cT_{n_k}$
as in the definition of tiling entropy has cardinality at most $e^{\delta |F|}$.
Take a set $R$ of representatives in $\overline{\Gamma A_k}$ for the relation of
having the same restriction to $F$. We may assume that $R\subseteq \Gamma A_k$ via perturbation.
In view of (i) we then have
\begin{align*}
\log |\pi_F (R)|
\leq \log (e^{\delta |F|} q^{(\theta + 1/2^k )|F|})
= \delta |F| + (\theta + 1/2^k )(\log q)|F| .
\end{align*}
Taking a F{\o}lner sequence of such $F$, we deduce that the topological entropy
of the action $\Gamma\curvearrowright\overline{\Gamma A_k}$ is less than
$\delta + (\theta + 1/2^k )\log q$.
Since $X\subseteq \overline{\Gamma A_k}$ it follows by the monotonicity of entropy
that $\htopol (\Gamma,X) \leq \delta + (\theta + 1/2^k )\log q$.
Since we can take $\delta$ arbitrarily small and $k$ arbitrarily large,
we infer that $\htopol (\Gamma,X) \leq \theta\log q = \lambda$
and hence that $\htopol (\Gamma,X) = \lambda$, as desired.

As we now know the topological entropy to be nonzero, we can deduce topological freeness from Lemma~\ref{L-top free}
using the fact that the action is minimal and $\Gamma$ is torsion-free.

Finally we verify that, given a nonempty open set $U\subseteq X$,
the free product $\Z_2 * \Z_2 * \Z_2$ embeds into the rigid stabilizer $\fA (\alpha )_U$.
Since $\Gamma A$ is dense in $X$ there is an $s\in\Gamma$ such that
$A\cap sU \neq\emptyset$, and as $\fA (\alpha )_{sU} = s\fA (\alpha )_U s^{-1}$ we may thus assume that $A\cap U \neq\emptyset$.

Fix a $z\in A\cap U$. Then there is a finite set $F\subseteq\Gamma$ such that the clopen set $\{ x\in X : x|_F = z|_F \}$
is contained in $U$. Fix a generator $a$ of the group $H$. By the definition of property ID, there exist $m,k\in\N$
and a shape $S$ of $\cT_{n_k}$ containing $F\cup \{ 1_\Gamma \}$ such that
the sets $Sa^{jm}$ for $j\in\Z$ are tiles of $\cT_{n_k}$ whose union contains $H$.
Let $l\in \{ 1,2,3 \}$. Write $V_l$ for the set of all $x\in X$ such that
\begin{itemize}
\item $x_{a^{-m}} = l$,

\item $x_{a^m}$, $x_{a^{3m}}$, and $x_{a^{5m}}$ are distinct elements of $\{1,2,3,4 \} \setminus \{ l \}$,

\item $x_{sa^{2jm}} = z_s$ for all $s\in S$ and $j = 0,1,2,3$.
\end{itemize}
Then $V_l$, $a^{2m} V_l$, $a^{4m} V_l$, and $a^{6m} V_l$ are pairwise disjoint clopen subsets of $X$ which are contained in $U$
and so we can define an element $g_l$ of $\fA (\alpha )_U$ by setting
\begin{itemize}
\item $g_l x = a^{6m} x$ for all $x\in V_l$,

\item $g_l x = a^{-6m} x$ for all $x\in a^{6m} V_l$,

\item $g_l x = a^{2m} x$ for all $x\in a^{2m} V_l$,

\item $g_l x = a^{-2m} x$ for all $x\in a^{4m} V_l$,

\item $g_l x = x$ for all $x\in X\setminus (V_l \sqcup a^{2m} V_l \sqcup a^{4m} V_l \sqcup a^{6m} V_l )$.
\end{itemize}
Each $g_l$ for $l=1,2,3$ generates a copy of $\Z_2$ inside of $\fA (\alpha )$
(each is the image of an even permutation under an embedding of the ordinary alternating group $\fA_4$ into
the dynamical alternating group $\fA_4 (\alpha )$
of the kind that appears in the definition of the latter, and $\fA_4 (\alpha )$ is contained in $\fA (\alpha )$
according to the comment after Definition~\ref{D-alternating}).
We will check that these three elements together generate a copy of $\Z_2 * \Z_2 * \Z_2$.

Let $r\in\N$ and let $(l_1 , \dots , l_r )$ be a tuple in $\{ 1,2,3 \}^r$ such that
$l_{j+1} \neq l_j$ for each $j=1,\dots ,r-1$. As $z\in A$ and each $A_k$ is a box,
by (iii), (iv), and (v) we can find an $x\in X$ such that
\begin{itemize}
\item $x_{a^{(6j-1)m}} = l_{j+1}$ for every $j=0,\dots , r-1$,

\item $x_{a^{(6j+1)m}}$, $x_{a^{(6j+3)m}}$, and $x_{a^{(6j+5)m}}$ are distinct elements of $\{ 1,2,3,4 \} \setminus \{ l_{j+1} \}$
for every $j=0,\dots , r-1$,

\item $x_{sa^{6jm}} = z_s$ for every $s\in S$ and $j=0,\dots , r$,

\item $x_{a^{(6r-1)m}} \neq x_{a^{-m}}$.
\end{itemize}
Then for each $j=0,\dots ,r-1$ we have $g_{l_j} g_{l_{j-1}} \cdots g_{l_1} x \in V_{l_{j+1}}$ and
hence $g_{l_r} g_{l_{r-1}} \cdots g_{l_1} x = a^{6rm} x \neq x$, so that
$g_{l_r} g_{l_{r-1}} \cdots g_{l_1}$ is not the identity element in $\fA (\alpha )$.
We conclude that the subgroup generated by $g_1$, $g_2$, and $g_3$
is isomorphic to $\Z_2 * \Z_2 * \Z_2$.
\end{proof}

\begin{remark}
If in the conclusion of the above theorem we drop the part about the rigid stabilizers
then we do not need to assume that $\Gamma$ has property ID.
That is, for every torsion-free countable amenable group $\Gamma$ and $\lambda\in (0,\infty )$
there exists an expansive topologically free minimal action
$\alpha$ of $\Gamma$ on the Cantor set with topological entropy $\lambda$.
\end{remark}

If a countable group $G$ of homeomorphisms of a Hausdorff space $X$ has the property
that for every nonempty open set $U\subseteq X$ the rigid stabilizer $G_U$ is nonamenable,
then every nontrivial uniformly recurrent subgroup of $G$ is nonamenable \cite{LeBMat18},
which implies by \cite{Ken15} that $G$ is C$^*$-simple (Corollary~1.3 of \cite{LeBMat18}).
Given that the free product $\Z_2 * \Z_2 * \Z_2$ contains a free subgroup on two generators 
(take for instance the product of the generators of the first two copies of $\Z_2$ and the 
product of the generators of the second two copies) we conclude that 
the alternating groups in Theorem~\ref{T-nonamenable} are all C$^*$-simple.

By Theorem~\ref{T-alt coe}
the dynamical alternating group is a complete invariant for continuous orbit equivalence
among topologically free minimal actions.
Since a group with property ID cannot be virtually cyclic,
Theorems~\ref{T-nonamenable} and \ref{T-coe} and the conclusion of the
previous paragraph then combine to yield the following.
All of the alternating groups below are simple by Theorem~\ref{T-simple fg}(i),
have property Gamma by Theorem~\ref{T-Gamma}, and are finitely generated
by Theorem~\ref{T-simple fg}(ii) (since $\Gamma$ is assumed to be finitely generated).

\begin{theorem}\label{T-ent nonamenable}
Suppose that $\Gamma$ is a finitely generated torsion-free amenable group with property ID.
Then there is an uncountable family
of topologically free expansive minimal actions $\alpha$ of $\Gamma$ on the Cantor set
such that $\fA (\alpha )$ is C$^*$-simple
and such that the groups $\fA (\alpha )$ for different $\alpha$ are pairwise nonisomorphic.
\end{theorem}

In view of Theorem~\ref{T-res finite} and Example~\ref{E-product} we obtain the following corollary.

\begin{corollary}
Suppose that $\Gamma$ is a finitely generated amenable group which is either (i) torsion-free, ICC, and residually finite
or (ii) of the form $\Gamma_0 \times \Z$ where $\Gamma_0$ is nontrivial and torsion-free.
Then there is an uncountable family
of topologically free expansive minimal actions $\alpha$ of $\Gamma$ on the Cantor set
such that $\fA (\alpha )$ is C$^*$-simple
and such that the groups $\fA (\alpha )$ for different $\alpha$ are pairwise nonisomorphic.
\end{corollary}

Examples of finitely generated amenable groups which are torsion-free, ICC, and residually finite are the wreath product
$\Z\wr\Z$ and any finitely generated torsion-free weakly branch group, such as the basilica group.

\end{document}